\documentclass[a4paper,11pt,oneside]{amsart}
\usepackage[utf8x]{inputenc}
\usepackage[english]{babel}

\usepackage{geometry}
\usepackage[bottom]{footmisc}

\usepackage{isomath}
\usepackage{amsmath,amsfonts,amssymb,amsthm}
\usepackage{upgreek}

\usepackage{enumitem}

\usepackage{float}

\usepackage{graphicx}

\usepackage{tabulary}

\usepackage{tensor}

\usepackage{url,epsfig}
\usepackage {lettrine}
\usepackage {epigraph}
\usepackage {sidecap}	      
\usepackage {caption}
\usepackage{lmodern}
\usepackage{booktabs}

\usepackage{csquotes}
 
\usepackage{bm} 

\usepackage{stmaryrd}

\usepackage{mathtools}
\usepackage{comment}
\usepackage{bbm} 
\usepackage{latexsym}

\usepackage{enumitem}

\usepackage[ocgcolorlinks, linkcolor=blue, citecolor=cyan, backref=page]{hyperref} 

\usepackage{mathrsfs}   

\usepackage{dsfont}

\usepackage{tikz-cd}

\usepackage[textsize=tiny,textwidth=0.8in]{todonotes}

\theoremstyle{plain}
\newtheorem{theorem}{Theorem}[section]
\newtheorem{proposition}[theorem]{Proposition}
\newtheorem{lemma}[theorem]{Lemma}
\newtheorem{corollary}[theorem]{Corollary}

\newtheorem{conjecture}[theorem]{Conjecture}
\newtheorem*{theorem*}{Theorem}
\newtheorem*{proposition*}{Proposition}
\newtheorem*{lemma*}{Lemma}
\newtheorem*{corollary*}{Corollary}
\newtheorem*{property*}{Properties}
\newtheorem*{conjecture*}{Conjecture}

\theoremstyle{definition}
\newtheorem{definition}[theorem]{Definition}

\newtheorem{remark}[theorem]{Remark}
\newtheorem*{definition*}{Definition}
\newtheorem*{example*}{Example}
\newtheorem*{remark*}{Remark}

\theoremstyle{remark}

\newtheorem*{note*}{Remark}
\newtheorem*{exercise*}{Esercizio}

\newcommand{\m}[1]{\mathcal{#1}}
\newcommand{\bb}[1]{\mathbb{#1}}
\newcommand{\mbf}[1]{\mathbf{#1}}
\newcommand{\mrm}[1]{\mathrm{#1}}
\newcommand{\msc}[1]{\textsc{#1}} 
\newcommand{\f}[1]{\mathfrak{#1}} 
\newcommand{\scr}[1]{\mathscr{#1}}

\newcommand{\idmatrix}{\mathbbm{1}}   
\newcommand{\del}{\partial}         
\newcommand{\delbar}{\bar{\partial}} 
\newcommand{\dd}{\mrm{d}}
\newcommand{\acts}{\curvearrowright} 

\newcommand{\Aut}{\mathrm{Aut}}

\DeclarePairedDelimiter\norm{\lVert}{\rVert} 
\DeclarePairedDelimiter{\set}{\{}{\}}

\newcommand{\ext}[1]{\bigwedge\nolimits^{#1}}

\numberwithin{equation}{section} 
\author{Annamaria Ortu}
\title[Optimal symplectic connections]{Optimal symplectic connections and deformations of holomorphic submersions}
\date{}
\email{aortu [at] sissa [dot] it}

\begin{document}

\begin{abstract}
We give a general construction of extremal K\"ahler metrics on the total space of certain holomorphic submersions, extending results of Dervan-Sektnan, Fine, and Hong. We consider submersions whose fibres admit a degeneration to K\"ahler manifolds with constant scalar curvature, in a way that is compatible with the fibration structure. Thus we allow fibres that are K-semistable, rather than K-polystable; this is crucial to moduli theory. On these fibrations we phrase a partial differential equation whose solutions, called \emph{optimal symplectic connections}, represent a canonical choice of a relatively K\"ahler metric. We expect this to be the most general construction of a canonical relatively K\"ahler metric provided all input is smooth.
We use the notion of an optimal symplectic connection and the geometry related to it to construct K\"ahler metrics with constant scalar curvature and extremal metrics on the total space, in adiabatic classes.
\end{abstract}

\maketitle



\section{Introduction}
Let $\pi: (X,H) \to (B,L)$ be a holomorphic submersion of a relatively polarised compact K\"ahler manifold onto a compact K\"ahler base. We address the problem of finding conditions under which the total space $X$ admits an extremal metric in the \emph{adiabatic classes}
\begin{equation*}
c_1(H) + kc_1(L) \quad \textrm{for} \ k\gg 0.
\end{equation*}
In general, one expects that such conditions reflect the geometry of the fibres and the geometry of the base, taking into account possible automorphisms and moduli of the fibres.
When all the fibres of $\pi$ admit a constant scalar curvature K\"ahler (cscK) metric, the problem was solved by Dervan and Sektnan in \cite{DervanSektnan_OSC1}, building on previous results by Hong \cite{Hong_cscK} and Fine \cite{Fine_cscK_fibrations} in more special situations.
Here we extend their result to more general fibrations, whose fibres are K-semistable analytic deformations of cscK fibres. We expect ours to be the most general situation where it is possible to construct extremal metrics in an adiabatic limit, provided all data considered is smooth. In particular, the condition we phrase on the fibration, called the \emph{optimal symplectic connection condition}, should mirror the Hermite-Einstein condition for vector bundles, and lead to the construction of a moduli space of holomorphic submersions and to a link with an algebro-geometric notion of stability.

The easiest and most instructive case to understand the ingredients involved is indeed the construction of constant scalar curvature K\"ahler metrics on the total space of \emph{projectivised vector bundles}, studied by Hong in \cite{Hong_cscK}. Let $\m{E} \to (B,L)$ be a holomorphic simple vector bundle, endowed with a Hermitian metric $h$. Let $\bb{P}(\m{E}) \to B$ be its projectivisation, obtained by taking over each $b \in B$ the projective space $\bb{P}({\m{E}}_b)$. Then $h$ induces a Hermitian structure $h^\vee$ on the hyperplane bundle $\m{O}_{\bb{P}(\m{E})}(1)$; the curvature form $F_{h^\vee}$ is such that $\omega_h := iF_{h^\vee}$ restricts to the Fubini-Study metric on each fibre of $\bb{P}(\m{E})$, which is cscK (in fact K\"ahler-Einstein).

If the Hermitian metric $h$ satisfies the Hermite-Einstein condition
\begin{equation*}
\Lambda_{\omega_B} F_h = \lambda\mathbbm{1},
\end{equation*}
then it is uniquely determined by the equation, which implies that there is a canonical choice of the Fubini-Study metric on the fibres of $\bb{P}(\m{E})\to B$. This choice allowed Hong to construct constant scalar curvature K\"ahler metrics on $\bb{P}(\m{E})$ in each adiabatic class $O_{\bb{P}(\m{E})}(1) + k L$ for all $k \gg 0$.

\subsection*{Optimal symplectic connections}
In the more general case of a polarised fibration with cscK fibres, $\pi: (X,H) \to (B,L)$, Dervan and Sektnan introduced in \cite{DervanSektnan_OSC1} a condition analogous to the Hermite-Einstein condition for projectivised vector bundles: the \emph{optimal symplectic connection} condition.
Let $\omega_X \in c_1(H)$ be a relative symplectic form on $X$ which restricts to a K\"ahler metric with constant scalar curvature on each fibre. Such 2-form is called a \emph{symplectic connection} in the language of symplectic fibrations because it determines a splitting of the tangent bundle of $X$ into a vertical and a horizontal part
\begin{equation*}
TX = \m{V}\oplus \m{H}^{\omega_X},
\end{equation*}
where $\m{H}_{\omega_X}$ is defined using orthogonality with respect to $\omega_X$. If one assumes that the fibres have a cscK metric, then these metrics can be used to construct a relatively cscK metric $\omega_X$ on $X$, but such an $\omega_X$ is not unique if the fibres have nontrivial automorphisms.
An optimal symplectic connection is then a preferred choice of $\omega_X$, defined in terms of a solution to a second-order elliptic PDE on the real vector bundle $E\to B$ of relatively cscK metrics. This choice allows one to consider a \emph{canonical} relatively cscK metric on the fibres.

In this work we extend further their result to the following setting. Let $(Y, H_Y) \to B$ be a holomorphic submersion and assume that the fibres are \emph{analytically K-semistable} manifolds, i.e. they each admit a degeneration to a cscK manifold. We assume also that these degenerations vary holomorphically in $B$, so that we have a degeneration $(\m{X}, \m{H}) \to B \times S$ of $(Y, H_Y) \to B$ to a fibrewise cscK fibration $(X, H) \to B$ parametrised by $S \in \bb{C}$, with a $\bb{C}^*$-action on $B \times S$ which lifts equivariantly to $(\m{X}, \m{H})$.
Using a relative version of Ehresmann's theorem (Proposition \ref{Prop:Ehresmann_relative}) we take the perspective of varying the complex structure of the underlying symplectic fibration, from a relatively cscK complex structure $J_0$ to small compatible deformations $J_s$ which keep $\pi$ holomorphic.
We also assume that the complex structure $J_B$ of the base is fixed, so the deformations we consider are just vertical.

Thus we can start by focusing on a single fibre, for which we rely on the deformation theory of cscK manifolds developed by Sz{\'e}kelyhidi \cite{Szekelyhidi_deformations} and Br\"onnle \cite{Bronnle_PhDthesis}, based in turn on the moment map interpretation of scalar curvature introduced by Fujiki \cite{Fujiki_momentmap} and Donaldson \cite{Donaldson_momentmap}. More precisely, let $(M, \omega, J_0)$ be a cscK manifold and let $\scr{J}$ be the infinite dimensional complex manifold of (almost) complex structures on $M$ which are compatible with $\omega_X$. The scalar curvature map
\begin{equation*}
J \mapsto \mrm{Scal}(\omega, J)-\widehat{S}
\end{equation*}
is an infinite dimensional moment map for the pull-back action of the group $\scr{G}$ of Hamiltonian symplectomorphisms on $\scr{J}$. In \cite{Szekelyhidi_deformations}, Sz{\'e}kelyhidi proves an analogue of Luna's slice theorem for the action $\scr{G} \acts \scr{J}$: he considers the Kuranishi space $V$ of compatible (almost) complex structures close to $J_0$, which is an open ball in a finite dimensional vector space and it parametrises the complexified orbits of $\scr{G}$ via an embedding $\Phi : V \hookrightarrow \scr{J}$. He proves that, in a neighborhood of the K-polystable point $J_0$, the scalar curvature can be reduced to a \emph{finite dimensional} moment map for the action of $K = \mrm{Stab}_{\scr{G}}(J_0)$ on $V$. Its second derivative is the moment map for the linearised action of $K$ on $T_0V$
\begin{equation}\label{Eq:map_nu_intro}
\nu : T_0V \to \f{k}.
\end{equation}

Going back to the fibration setting, for each $b \in B$ the Lie algebra $\f{k}_b$ can be identified with the fibre $E_b$ of the vector bundle of fibrewise cscK metrics. Since the deformations we are considering are vertical, in \eqref{Eq:nu_section} we define a smooth section $\nu$ of $E \to B$ which on each fibre is the map \eqref{Eq:map_nu_intro}, which therefore encodes the deformation of the complex structure.
We say that $\omega_X$ is an \emph{optimal symplectic connection} on $(Y, H_Y)$ if
\begin{equation*}
p_E(\Updelta_{\m{V}}(\Lambda_{\omega_B} (m^*F_{\m{H}})) + \Lambda_{\omega_B} \rho_{\m{H}}) + \frac{\lambda}{2}\nu = 0.
\end{equation*}
In this expression $F_{\m{H}}$ and $\rho_{\m{H}}$ are curvature quantities which depend on $\omega_X$ and $J_0$, $p_E$ is the projection onto the global sections of the vector bundle $E \to B$ computed with respect to $J_0$, and $\lambda >0$. The vanishing of the first term is the condition for an optimal symplectic connection in the sense of \cite{DervanSektnan_OSC1}, i.e. where all the fibres are cscK, so our notion generalises their notion.
This seems to be the first geometric PDE in complex geometry which involves both curvature quantities and the change in complex structure.

We expect optimal symplectic connections to be unique, as happens in the relatively cscK case \cite{DervanSektnan_OSC3, Hallam_geodesics}, up to the action of a suitable subset of the automorphisms of $Y$ which preserves the projection $\pi_Y$. In this way, one can genuinely call an optimal symplectic connection a \emph{canonical choice} of a relatively K\"ahler metric on a relatively K-semistable fibration.

\subsection*{Extremal metrics on the total space of holomorphic submersions}
We make use of optimal symplectic connections to prove the existence of cscK and extremal metrics on the total space $Y$. In order to choose an appropriate metric on the base manifold, we will need some moduli theory (developed in \cite{Fujiki_Schumacher_Moduli_cscK,DervanNaumann_ModuliCscK}). Let $\m{M}$ be the moduli space of cscK manifolds and let $q : B \to \m{M}$ be the moduli map induced by the central family $(X, H) \to B$, which fibres are cscK. $\m{M}$ can be endowed with a \emph{Weil-Petersson} K\"ahler metric, and we denote by $\alpha_{WP}$ the pull-back of it via $q$. This is a smooth semi-positive $(1,1)$-form on $B$.

We first consider the case where the group of automorphisms of $(Y, H_Y)$ and of $(B,L)$ (or rather of the map $q$) are discrete. Thus we require that the base admits a \emph{twisted cscK metric} with twisting form $\alpha_{WP}$:
\begin{equation*}
\mrm{Scal}(\omega_B) - \Lambda_{\omega_B}\alpha_{WP} = c.
\end{equation*}

\begin{theorem}\label{Thm:1}
Let $\omega_X$ be an optimal symplectic connection and $\omega_B$ be a twisted cscK metric with twisting $\alpha_{WP}$. Then there exists a constant scalar curvature K\"ahler metric on $Y$ in the class $c_1(H_Y) + kc_1(L)$ for all $k \gg 0$.
\end{theorem}

If we allow the moduli map $q$ of the central fibration and the total space $(Y, H_Y)$ to have automorphisms, the adiabatic limit method produces \emph{extremal metrics} on the total space. In this case, we have to modify our hypothesis on $\omega_X$ and $\omega_B$ as follows: we require that $\omega_B$ is \emph{twisted extremal}, i.e.
\begin{equation*}
\mrm{Scal}(\omega_B) - \Lambda_{\omega_B}\alpha_{WP}
\end{equation*}
is a holomorphy potential on $B$ and that $\omega_X$ is an \emph{extremal symplectic connection}, i.e.
\begin{equation*}
p_E(\Updelta_{\m{V}}(\Lambda_{\omega_B} (m^*F_{\m{H}})) + \Lambda_{\omega_B} \rho_{\m{H}}) + \frac{\lambda}{2}\nu
\end{equation*}
is a holomorphy potential on $Y$. We also need some technical assumptions which we will explain in \S \ref{Subsec:approx_sol_extremal}: the group of automorphisms of $\pi_Y$ acts equivariantly on the family $\m{X} \to B \times S$ and $\omega_X$ is invariant under the flow of the extremal vector fields.
\begin{theorem}\label{Thm:2}
Suppose that $(B, L)$ admits a twisted extremal metric $\omega_B$ and $(Y, H_Y)$ admits an extremal symplectic connection $\omega_X$. Suppose also that all automorphisms of the moduli map $q$ lift to $(Y, H_Y)$. Then there exists an extremal metric on $Y$ in the class $c_1(H_Y) + kc_1(L)$ for all $k \gg 0$.
\end{theorem}

Our result generalises previous works by many authors which consider more special situations: we already mentioned Hong's paper \cite{Hong_cscK} about cscK metrics on the projectivisation of holomorphic vector bundles, in the case of discrete group of automorphism.
When the fibres admit moduli, Fine \cite{Fine_cscK_fibrations} proved the existence of cscK metrics on the total space of a fibration where all the fibres and the base are Riemann surfaces of genus $g\ge 2$.
In this case, the choice of a relatively K\"ahler metric on the total space falls naturally on the hyperbolic metric, and the optimal symplectic connection condition is vacuous. Again on the projectivisation of vector bundles, Br\"onnle \cite{Bronnle_extremal_projvb} proved the existence of extremal metrics in the presence of automorphisms. In both the cases of projectivised vector bundles, the relevant condition to require in order to produce such special metrics is the Hermite-Einstein condition on the corresponding vector bundle.
It has been proved by Dervan-Sektnan \cite{DervanSektnan_OSC1} that the optimal symplectic connection condition reduces to the Hermite-Einstein condition on projectivised vector bundles, thus being a genuine generalisation.
Moreover, they prove Theorems \ref{Thm:1} and \ref{Thm:2} in the case of a relatively cscK fibration.
Our notion of an optimal symplectic connection on a relatively K-semistable fibration should be the most general condition to ask in order to produce cscK or extremal metrics with an adiabatic limit technique, provided all data is smooth and the aforementioned hypotheses on the lifting of the automorphism groups hold.


The proof of the theorems is carried out using the \emph{adiabatic limit} technique, a strategy which originates in K\"ahler geometry in the work of Fine \cite{Fine_cscK_fibrations}.
Even if the situation he considers is different, some of the analytic results are general. In the prior work, the adiabatic limit strategy consists of expanding the scalar curvature of $\omega_X + k \omega_B$ in inverse powers of $k$, with the idea that if $k$ is large the base becomes very big and the curvature is concentrated in the vertical direction.
In the easiest case of discrete automorphisms, the optimal symplectic connection condition and the twisted cscK equation on the base allow one to find a relatively K\"ahler metric which is constant to order $k^{-1}$. Then one proceeds inductively, adding at each step $r$ a potential $i\del\delbar\varphi_r$ in order to make the scalar curvature constant up to the $k^{-r-1}$-term.
The implicit function theorem then allows one to deform the approximate solution to a genuine solution.

Our approach is a version of the one just described, except with \emph{two} parameters. We consider a degeneration $\m{X} \to B \times S$ of the fibration $Y \to B$ to the relatively cscK fibration $X \to B$ and we expand the scalar curvature of the K\"ahler metric $(\omega_X + k \omega_B, J_s)$ in inverse powers of $k$ and powers of $s$. Then we relate the parameters $k$ and $s$ by imposing $\lambda k^{-1} = s$, for some $\lambda >0$. We understand the linearisation of the optimal symplectic connection equation by proving a relative version of Kuranishi's Theorem (see \ref{Thm:relative_Kuranishi}), then we can apply the adiabatic limit technique as in previous works.

\subsection*{Outlook}
The optimal symplectic connection equation has a deep algebro-geometric meaning, coming from the fact that it can be viewed as a generalisation of the Hermite-Einstein condition for vector bundles, as well as a way to study how polarised varieties vary in families. So a natural question one could ask is whether optimal symplectic connections are linked to moduli problems of polarised varieties.

In the case of vector bundles, one can form a moduli space of \emph{stable} holomorphic vector bundles, and the Hitchin-Kobayashi correspondence of Donaldson-Uhlenbeck-Yau \cite{Donaldson_HYMstability, UhlenbeckYau_HYMstability} states that a holomorphic vector bundle on a compact K\"ahler manifold is stable if and only if it admits a Hermite-Einstein connection.

For polarised K\"ahler manifolds, there is an algebraic notion of stability, \emph{K-stability}, which is conjecturally equivalent by the Yau-Tian-Donaldson conjecture to the existence of cscK metrics \cite{Yau_OpenProblems, Tian_KahlerEinstein, Donaldson_1StabilityToric}. It has been proved that it is possible to form an analytic moduli space of cscK manifolds \cite{Fujiki_Schumacher_Moduli_cscK, DervanNaumann_ModuliCscK}; see also \cite{Inoue_ModuliSpaceFano}.

Motivated by these results, Dervan, Sektnan \cite{DervanSektnan_OSC2} and Hallam \cite{Hallam_geodesics} introduced and studied a notion of \emph{stability of fibrations} which on one hand is a generalisation of slope-stability for vector bundles, and on the other hand is a generalisation of K-stability for K\"ahler manifolds. Though they only consider the case of stability of fibrations when the fibres admit cscK metrics, the notion of stability also makes sense in the case where the fibres are just K-semistable. This is sharp: the notion is not reasonable when the fibres are $K$-unstable, giving more evidence of the fact that the situation we consider is the most general possible in the smooth case.

One can also take the analytic perspective, and try to form a moduli space of fibrations that admit an optimal symplectic connection. We conjecture that our generalisation of the optimal symplectic connection condition leads to the existence of such moduli space. In particular, we expect that the optimal symplectic connection condition is open in families of holomorphic submersions with discrete automorphisms.

\begin{conjecture}
There exists a Hausdorff complex space which is a moduli space for holomorphic fibrations which admit an optimal symplectic connection.
\end{conjecture}

When deforming a fibrewise cscK fibration one cannot expect the fibres to remain cscK, unless one requires that the fibres are rigid, so the setting considered by Dervan-Sektnan is not the right one to be related to moduli theory. Our construction, though, allows the fibres to be K-semistable, which is an open condition, thus it should be possible to study the local behaviour in families of fibrations with optimal symplectic connections without restricting to the rigid case.
The first step in this direction would be proving the openness of the optimal symplectic connection condition. This should also lead to new examples, produced by starting with a relatively cscK fibration admitting an optimal symplectic connection and applying the implicit function theorem to obtain an optimal symplectic connection on a deformed relatively K-semistable fibrations. We plan to return to this in future work.

While our work settles the problem in the case all input is smooth, open questions remain in the presence of singularities. For instance, one could consider holomorphic fibrations where the total space is smooth but some fibres are singular. In \cite[\S 9]{Fine_PhDThesis} Fine explains a possible way to obtain a similar existence result for special K\"ahler metrics on holomorphic Lefschetz fibrations, where a finite number of fibres are singular, but the problem is still mostly open.
Moreover, denoting $\overline{\f{h}}_b$ the Lie algebra of holomorphic vector fields admitting a holomorphy potential on the fibre $X_b$ of the relatively cscK degeneration, an important technical hypothesis in our approach is that its dimension is independent on $b$. This is a smoothness hypothesis that allows defining the vector bundle $E \to B$ of relatively cscK metrics on $X \to B$, and removing it would require a different approach to our problem.
Finally, a key assumption in Theorem \ref{Thm:2} is that all automorphisms of the moduli map on the base lift to the total space. When this assumption does not hold, existence results for special K\"ahler metrics were proved by Hong \cite{Hong_stability_ruled_mfds} in the case of projectivised vector bundles and extended by Lu-Seyyedali \cite{LuSeyyedali_extremal}, but the problem is open on a general fibration, and even for projective bundles sharp results are not known.

\subsection*{Outline} In Section \ref{Sec:preliminaries} we give preliminary definitions and results about the moment map interpretation of scalar curvature, due to Fujiki \cite{Fujiki_momentmap} and Donaldson \cite{Donaldson_momentmap}. Then we recall the relevant definitions and results on deformations of a cscK manifold, following \cite{Szekelyhidi_deformations}. In Section \ref{Sec:optimal symplectic connection} we recall the definition of an optimal symplectic connection in the relatively cscK case and we extend it to the relatively K-semistable case. In Section \ref{Sec:deformations_of_fibrations} we extend the theory of deformations of a cscK manifold to the fibration setting and we prove a relative version of Kuranishi's Theorem. In Section \ref{Sec:adiabatic_limit} we prove the existence of a cscK metric on the total space of the relatively K-semistable fibration: we derive the optimal symplectic connection equation by expanding the scalar curvature, then we study its linearisation and we use an adiabatic limit as in \cite{Fine_cscK_fibrations} to prove the existence of cscK metrics.

\subsection*{Acknowledgments}
I am very grateful to my supervisors: to Ruadha\'i Dervan for encouraging me to study this problem and for sharing his insights with me, and to Jacopo Stoppa for many helpful discussions and constant support. I also thank Carlo Scarpa and Lars Martin Sektnan for enlightening comments and suggestions.
Part of this work was performed while visiting the University of Cambridge, funded by Ruadha\'i Dervan's Royal Society University Research Fellowship.

\section{Preliminaries}\label{Sec:preliminaries}
Let $(M, L)$ be a polarised compact K\"ahler manifold of dimension $n$, and let $\omega \in c_1(L)$ be a K\"ahler form. The \emph{scalar curvature} of $\omega$ is the function on $M$ defined as the contraction of the Ricci curvature:
\begin{equation*}
\mrm{Scal}(\omega) = \Lambda_{\omega} \mrm{Ric}(\omega, J),
\end{equation*}
where
\begin{equation*}
\mrm{Ric}(\omega, J) = -\frac{i}{2\pi} \del_J \delbar_J \mrm{log} \ \omega^n.
\end{equation*}
We will be interested in K\"ahler metrics with constant scalar curvature, where the constant is given by the average scalar curvature, is a topological constant and it is given by
\begin{equation*}
\widehat{S} = \frac{n \ c_1(M) \cdot c_1(L)^{n-1}}{c_1(L)^n}.
\end{equation*}
In this section, we quickly recall some basic definitions and results on K\"ahler manifolds and the scalar curvature map. See \cite[Chapter 4]{Szekelyhidi_book} for exhaustive discussions. For details and proofs about the moment map picture of the scalar curvature see \cite[Chapter 1]{Scarpa_HcscKSystem}.

\subsection{Extremal K\"ahler metrics}
Let $(M, L)$ be a K\"ahler manifold with K\"ahler structure $(\omega, J)$.  Denote by $g_J$ the Riemannian metric on $M$ induced by $J$ and $\omega$: $g_J(\cdot, \cdot) = \omega(\cdot, J\cdot)$.
Recall that a function $h \in C^\infty(M)$ is called a \emph{holomorphy potential} if the $(1,0)$-part of the Riemannian gradient of $h$, denoted $\nabla_g^{1,0}h$, is a holomorphic vector field.
\begin{definition}
A K\"ahler metric $(\omega, J)$ on $M$ is \emph{extremal} if
\begin{equation*}
\delbar \nabla_g^{1,0}\mrm{Scal}(\omega) = 0,
\end{equation*}
i.e. if the scalar curvature of $g$ is a \emph{holomorphy potential}.
\end{definition}
We denote by $\overline{\f{h}}$ the set of holomorphy potentials and by $\f{h}_0$ the set of holomorphic vector fields which admit an holomorphy potential. We denote by $\m{D} : C^\infty(M, \bb{C}) \to \Omega^{0,1}(T^{1,0}M)$ the operator
\begin{equation*}
\m{D}(\varphi) = \delbar \nabla^{1,0}_g \varphi.
\end{equation*}
The \emph{Lichnerowicz operator} is the composition $\m{D}^*\m{D}$, where the adjoint is defined with respect to the $L^2(g)$-inner product. It can be written explicitly as follows:
\begin{equation*}
\m{D}^*\m{D}(\varphi) = \Updelta^2_g(\varphi) + \langle \mrm{Ric}(\omega), i \del\delbar\varphi\rangle + \langle \nabla \mrm{Scal}(\omega), \nabla \varphi\rangle.
\end{equation*}
The Lichnerowicz operator is a 4th-order elliptic operator, and its kernel $\mrm{ker}\m{D}^*\m{D} = \mrm{ker}\m{D}$ coincides with holomorphy potentials on $M$. In particular $\omega$ is an extremal metric on $M$ if $\m{D}(\mrm{Scal}(\omega))=0$.

Let $\m{L}$ be the linearisation of the scalar curvature function. Then $\m{L}$ can be written in terms of the Lichnerowicz operator:
\begin{equation*}
\m{L}(\varphi) = -\m{D}^*\m{D}(\varphi) + \frac{1}{2}\langle \nabla \mrm{Scal}(\omega), \nabla \varphi \rangle.
\end{equation*}
In particular if the scalar curvature is constant, the linearisation is exactly the Lichnerowicz operator.
Solving the extremal equation means requiring
\begin{equation*}
\mrm{Scal}(\omega) - f = 0
\end{equation*}
for some holomorphy potential $f$. If we change $\omega$ to $\omega+ i\del\delbar\varphi$, then the holomorphy potential $f$ changes to $f+ \frac{1}{2}\langle \nabla f, \nabla \varphi \rangle$. Thus to find an extremal metric in the class $[\omega]$, we need to find a zero of the operator
\begin{equation}\label{Eq:extremal_operator}
\begin{aligned}
C^\infty(M, \bb{R}) \times \overline{\f{h}} &\to C^\infty(M, \bb{R}) \\
(\varphi, h) &\mapsto \mrm{Scal}(\omega+i\del\delbar\varphi) - \frac{1}{2}\langle \nabla f, \nabla \varphi \rangle-f.
\end{aligned}
\end{equation}
The linearisation $\m{G}$ of this operator at a solution is given again by the Lichnerowicz operator itself: $\m{G}(\varphi, 0) = -\m{D}^*\m{D}\varphi$.

\subsection{Scalar curvature as a moment map}
Let $(M, \omega)$ be a symplectic manifold, and consider the infinite-dimensional manifolds
\begin{equation*}
\scr{J} = \set*{J : TM \to TM \ \text{almost complex structure compatible with }\omega}.
\end{equation*}
The tangent space at a point $J$ is given by
\begin{equation*}
T_J\scr{J} = \set*{A:TM \to TM \left| JA+AJ=0 \ \text{and} \ \omega(u, Av) + \omega(Au, v) = 0\right.}.
\end{equation*}

Fix $J \in \scr{J}$ an integrable complex structure and $A \in T_J\scr{J}$.
Since $AJ+JA = 0$, $A$ maps $T^{1,0}M$ to $T^{0,1}M$ and $T^{0,1}M$ to $T^{1,0}M$, where the splitting is considered with respect to $J$. Since $A$ is real, it is uniquely determined by one of its restrictions. Identifying $A$ with $A: T^{0,1}M \to T^{1,0}M$ induces an identification
\begin{equation}\label{Eq:tangent_curlyJ}
\begin{aligned}
T_J\scr{J} \longleftrightarrow T_J^{0,1} \scr{J} = \set*{\left.\alpha \in \Omega^{0,1}(T^{1,0}M) \right| \omega(\alpha(u), v) + \omega(u, \alpha(x)) = 0}\\
\end{aligned}
\end{equation}
Now, if $A \in T_J \scr{J}$ then $JA \in T_J \scr{J}$, so $\scr{J}$ has an almost complex structure which is formally integrable \cite[\S9.2]{Gauduchon_Extremal_Introduction}. Moreover, it has a Hermitian inner product
\begin{align*}
\langle A, B \rangle_J := \int_M \langle A, B \rangle_{g_J} \frac{\omega^n}{n!},
\end{align*}
and the two combine to give a K\"ahler form, given at the point $J$ by
\begin{align*}
\bm{\Omega}_J (A, B) = \langle JA, B \rangle_J.
\end{align*}
So $\scr{J}$ is an infinite-dimensional K\"ahler manifold. We consider the complex submanifold $\scr{J}^{\mrm{int}}$ of \emph{integrable} almost complex structures of $\scr{J}$. Its tangent space is given by those $\alpha \in T_J^{0,1}\scr{J}$ such that $\delbar \alpha =0$.

Recall that a vector field $\xi$ is Hamiltonian with respect to $\omega$ if there exists $h \in C^\infty(M, \bb{R})$ such that $\omega (\xi, \cdot)= -\dd h $.
On a K\"ahler manifold $\xi = J\nabla^g(h)$, where $\nabla^gh$ is the Riemannian gradient of $h$. We will denote an Hamiltonian vector field with Hamiltonian $h$ by $\xi_h$.

Consider the group of Hamiltonian symplectomorphisms of $(M,\omega)$, denoted by $\scr{G}$. This is the infinite-dimensional Lie group of time-one flows of Hamiltonian vector fields on $M$, and it acts on $\scr{J}$ by pull-back: for $J \in \scr{J}$ and $\phi \in \scr{G}$ the action is given by
\begin{align*}
\phi^* J := \dd \phi^{-1} \circ J \circ \dd \phi.
\end{align*}

The Lie algebra of $\scr{G}$ can be identified with $C^\infty_0(M)$, the smooth functions on $M$ with $\omega$-average zero. The following theorem is due to Fujiki \cite{Fujiki_momentmap} and Donaldson \cite{Donaldson_momentmap}.
 \begin{theorem}
 The action $\scr{G} \acts \scr{J}$ is Hamiltonian with moment map
 \begin{equation}\label{Eq:moment_map_cscK}
 \begin{aligned}
 \mu: \scr{J} &\longrightarrow \mrm{Lie}(\scr{G})^* \\
 J &\longmapsto \mrm{Scal}(\omega,J) - \widehat{S}.
 \end{aligned}
 \end{equation}
 \end{theorem}
 Therefore cscK metrics on $M$ correspond to $J \in \scr{J}^{\mrm{int}}$ such that $\mu(J) = 0$.
 A few comments:
 \begin{enumerate}
 \item If $J$ is integrable, $\mrm{S}(\omega, J)$ is the scalar curvature of the metric $g_J$. Otherwise, it is the Hermitian scalar curvature of the Chern connection on $TM$, which is not the same as the Levi-Civita connection in general.
 \item $S(J)-\widehat{S}$ is viewed as an element of $C^\infty_0(M)$ by identifying $\mrm{Lie}(\scr{G})^*$ with its dual via the $L^2(\omega)$-product on $M$.
\end{enumerate}
%
%

Let $J$ be an integrable compatible complex structure on $(M,\omega)$, and denote with $\m{K}_J(\omega)$ the set of K\"ahler metrics in the same K\"ahler class of $\omega$ with respect to $J$, i.e. those metrics which can be written as $\omega + i \del_J \delbar_J \phi$ for some $\phi \in C^\infty(M, \bb{R})$. The following proposition (see also \cite[\S 6.1]{Szekelyhidi_book}) justifies the fact that instead of moving the K\"ahler metric inside the K\"ahler class with respect to a fixed $J$, one can think of $\omega$ as fixed and move the complex structure.

\begin{proposition}\label{Prop:parallel_cpx-symp_structure}
\cite[p.17]{Donaldson_Symmetric_spaces} For every $\omega_\phi \in \m{K}_J(\omega)$ there exist $f \in \mrm{Diff}_0(M)$ such that $f^*\omega_\phi = \omega$ and $(M, \omega_\phi, J)$ is isomorphic to $(M, \omega, f^*J)$.
\end{proposition}
Thus if we fix $J \in \scr{J}$ integrable, we can define a map
\begin{equation}\label{Eq:map_F}
\begin{aligned}
F : \set*{\phi \in C^\infty(M, \bb{R}) \left| \omega_\phi \in \m{K}_J(\omega)\right.} &\longrightarrow \scr{J}\\
\phi &\longmapsto F_{\phi}J := f^*J.
\end{aligned}
\end{equation}
The differential at $0$ of $F$ is given by
\begin{equation*}
\dd_0F(\phi) = J\m{L}_{\xi_{\phi}(\omega)}J.
\end{equation*}

We can think of the image of the map $F$ as an infinitesimal complexified orbit of $\scr{G}\acts \scr{J}$, even if a complexification of $\scr{G}$ does not genuinely exist. Proposition \ref{Prop:parallel_cpx-symp_structure} says that a variation of the K\"ahler form in a given K\"ahler class for $J$ fixed corresponds to a variation of the complex structure $J$ in the same $\scr{G}^c$-orbit, for $\omega$ fixed.

\subsection{Deformation theory of cscK manifolds}\label{Subsec:Sz_def_theory}
In this subsection we follow Sz{\'e}kelyhidi \cite[\S 3]{Szekelyhidi_deformations}; similar results were obtained also by Br\"onnle \cite[Part 1]{Bronnle_PhDthesis}. We fix an integrable complex structure $J_0 \in \scr{J}$ on $(M, \omega)$ such that $\mrm{Scal}(\omega, J_0)$ is constant. 
\begin{definition}\label{Def:map_P}
For $J \in \scr{J}$ fixed we define the operator
\begin{equation*}
\begin{aligned}
P : C^\infty_0(M) &\longrightarrow T_J\scr{J} \\
h &\longmapsto \m{L}_{\xi_h} J,
\end{aligned}
\end{equation*}
where $\xi_h$ is the Hamiltonian vector field with Hamiltonian function $h$.
\end{definition}
\begin{remark}
$P$ is the infinitesimal action of $\scr{G}$ on $\scr{J}$ and one can show that $P(h) = 2J \delbar \xi_h^{1,0} + 2J\overline{(\delbar \xi_h^{1,0})}$. Thus, through the identification of Equation \eqref{Eq:tangent_curlyJ}, an equivalent operator is $P(h) = \delbar \xi_h^{1,0}$.
\end{remark}

The deformations of the complex structure are encoded in a complex
\begin{align*}
C^\infty_0(M, \bb{C}) \overset{P^{\bb{C}}}{\to} T_{J_0}\scr{J} \overset{\delbar}{\to} \Omega^{0,2}(T^{1,0}M),
\end{align*}
where $P^{\bb{C}}$ is defined on a complex function $f+ih$ as $P(f)+JP(h)$. Denote by $\widetilde{H}^1$ the cohomology of the complex:
\begin{equation}\label{Eq:H_tilde_cscK}
\widetilde{H}^1 = \set*{\alpha \in \Omega^{0,1}(T^{1,0}M)| \delbar \alpha = 0 = P^*\alpha}.
\end{equation}
This is a finite dimensional vector space, since it is the kernel of the elliptic operator \cite[\S 2]{Fujiki_Schumacher_Moduli_cscK}
\begin{equation}\label{Eq:box_laplacian}
\square: = PP^*+(\delbar^*\delbar)^2
\end{equation}
on $T_{J_0}\scr{J}$. Consider the space of Hamiltonian isometries
\begin{equation*}
K := \mrm{Stab}_{\scr{G}}(J_0) = \set*{\varphi \in \scr{G} \left| \dd\varphi^{-1}\circ J_0\circ\dd\varphi = J_0\right.} = \mrm{Aut}(M, J:0) \cap \scr{G}.
\end{equation*}
The Lie algebra of $K$, denoted $\f{k}$, can be identified with the kernel of $P$; it consists of smooth functions over $M$ whose Hamiltonian vector field is holomorphic.
The complexification of $\f{k}$ is then given by the kernel of the Lichnerowicz operator, and $K^\bb{C} = \mrm{Aut}(M, L)$.
The group $K$ acts naturally on $\widetilde{H}^1$ by pull-back and $0$ is a fixed point of the action.
The following theorem is a symplectic version of Kuranishi's theorem \cite{Kuranishi_family_cpx_str}.

\begin{theorem}\label{Thm:Kuranishi}
\cite{Kuranishi_family_cpx_str,Szekelyhidi_deformations}
There exists a ball around the origin $V \subset \widetilde{H}^1$ and a $K$-equivariant map
\begin{equation}\label{Eq:Kuranishi_map}
\Psi : V \to \scr{J}
\end{equation}
such that $\Psi (0) = J_0$ and
\begin{enumerate}
\item the $\scr{G}^c$-orbit of every integrable complex structure near $J_0$ intersects the image of $\Psi$;
\item If $x, x' \in V$ are in the same orbit for the complexified action of $K$, and $\Psi(x)$ is integrable, then $\Psi(x), \Psi(x')$ are in the same $\scr{G}^c$-orbit.
\end{enumerate}
\end{theorem}

\begin{remark}
The Kuranishi space $V$ is called a \emph{local slice} of the $\scr{G}^c$-action near the reference complex structure $J_0$, and Theorem \ref{Thm:Kuranishi} is an infinite-dimensional version of Luna's slice theorem \cite{Luna_slice}. Since we allow also non integrable almost complex structure, the slice is an actual ball. Instead, the Kuranishi space described in \cite{Kuranishi_family_cpx_str}, which parametrises only \emph{integrable} complex structures, is an analytic subset of our $V$.
\end{remark}

Let $\Omega$ be the symplectic form on $V$ given by pulling back via $\Phi$ the K\"ahler form $\bm{\Omega}$ on $\scr{J}$. Then the scalar curvature induces a moment map for the $K$-action on $V$,
\begin{equation}\label{Eq:Moment_map_Sz}
\begin{aligned}
\mu : V &\to \f{k}\\
x &\mapsto p_{\f{k}}\left(S(\omega, \Psi(x))\right),
\end{aligned}
\end{equation}
where $\f{k}$ is identified with its dual via the $L^2$-product of functions and $p_{\f{k}}$ is the orthogonal projection.
\begin{corollary}\cite[\S 3]{Szekelyhidi_deformations}
Possibly after shrinking $V$, the map $\Psi$ can be perturbed to a map
\begin{equation}\label{Eq:Kuranishi_map_deformed}
\Phi : V \to \scr{J}
\end{equation}
such that for all $x \in V$, $S(\omega, \Phi(x)) \in \f{k}$.
\end{corollary}
Hence the moment map \eqref{Eq:Moment_map_Sz} can be defined as $\mu(x) = S(\omega, \Phi(x))$.

By identifying $T_0V$ with $\widetilde{H}^1$, we consider on $\widetilde{H}^1$ the linear symplectic form $\Omega_0 (\cdot, \cdot) = \bm{\Omega}_{J_0}(\dd_0\Phi \cdot, \dd_0 \Phi \cdot)$, and the linear action of $K$ induced by the one on $V$. Fix $f \in \f{k}$, and define an endomorphism of $\widetilde{H}^1$ by
\begin{equation*}
A_{f}(t) = \dd_0\left(y \mapsto \mrm{exp}(tf)\cdot y \right),
\end{equation*}
where $y \in V$ and by $\mrm{exp}(tf)$ we denote the 1-parameter subgroup of $K$ defined by the element $f \in \f{k}$ (which corresponds via $\Phi$ to the flow of the Hamiltonian vector field $\mrm{grad}^\omega f$ on $M$).
Let
\begin{equation} \label{Eq:linearised_infinitesimal_action}
A_f := \left.\frac{\dd}{\dd t}\right|_{t=0} A_f(t).
\end{equation}

We have the following properties:
\begin{enumerate}[label=(\roman*)]
\item $A_f$ is a skew-hermitian endomorphism of $(\widetilde{H}^1, J_0)$ and $A_f(t) = \mrm{exp}(tA_f)$;
\item For $v \in \widetilde{H}^1$, denote by $\mbf{v}$ a vector field on $V$ such that $\mbf{v}|_{0} = v$. Then:
\begin{equation}\label{Eq:property_of_A_and_Lie_der}
A_f(v) = \del_t|_{t=0} A_f(t) v = \del_t|_{t=0} \left( (\Phi_t^f)_*(\mbf{v}) \right)_0 = -(\m{L}_{X_f} \mbf{v})_0 = [\mbf{v}, \xi_f]_0.
\end{equation}
\end{enumerate}
\begin{definition}\label{Def:map_nu}
We define a map $\nu : \widetilde{H}^1 \to \f{k}$ by
\begin{equation*}
\langle \nu(v), f \rangle = \frac{1}{2}\Omega_0(A_fv, v).
\end{equation*}
\end{definition}

It can be characterised as a moment map by relating it to the scalar curvature \eqref{Eq:Moment_map_Sz} as follows (see \cite[\S 3]{Inoue_ModuliSpaceFano} for the proof).
\begin{proposition}\label{Prop:expansion_mu_nu}
The map $\nu$ is a moment map for the linear $K$-action on $\widetilde{H}^1$ and if $v_t$ is a path in $V$ with $\dot{v}_0 = v$,
\begin{equation*}
\mu(v_t) =\frac{t^2}{2}\nu(v) + O(t^3).
\end{equation*}
\end{proposition}

\section{Optimal symplectic connections}\label{Sec:optimal symplectic connection}

Let $(X,H) \overset{\pi}{\to} (B,L)$ be a holomorphic submersion, where:
\begin{enumerate}[label=(\roman*)]
\item $X$ and $B$ are compact K\"ahler manifolds;
\item $L \to B$ is an ample line bundle, so we have a K\"ahler metric $\omega_B \in c_1(L)$;
\item $H \to X$ is a relatively ample line bundle, i.e. $\left.H\right|_b \to X_b$ is ample. This means that we can consider $\omega_X \in c_1(H)$ relatively K\"ahler, i.e. a closed $2$-form which restricts to a K\"ahler metric on each fibre of $\pi$.
\end{enumerate}
By Ehresmann's fibration theorem all the fibres are diffeomorphic. We denote $\mrm{dim}(B) = n$, $\mrm{dim}(X_b) = m$, so that $\mrm{dim}(X) = n+m$.
We assume further that the complex dimension of the Lie algebra $\f{h}(X_b)$ of holomorphic vector fields on the fibre $X_b$ is independent on $b$.

\subsection{Splitting of the function space}\label{Subsec:splitting_function_space}
Assume that $\omega_X$ restricts to a \emph{constant scalar curvature} K\"ahler metric $\omega_b$ on each fibre $X_b$. Let
\begin{equation*}
\m{D}_{\m{V}}^*\m{D}_{\m{V}} : C^\infty(X, \bb{R}) \to C^\infty(X, \bb{R})
\end{equation*}
be the vertical Lichnerowicz operator, defined fibrewise as $\left( \m{D}_{\m{V}}^*\m{D}_{\m{V}} \varphi \right)|_{X_b} = \m{D}_b^*\m{D}_b \left. \varphi\right|_{X_b}$. It is a real operator since the fibrewise metric is cscK. By integrating a function $\varphi \in C^\infty(X, \bb{R})$ over the fibres of $\pi$, we define a projection
\begin{equation*}
\begin{aligned}
C^\infty(X, \bb{R}) &\longrightarrow C^\infty(B, \bb{R}) \\
\varphi &\longmapsto \int_{X/B} \varphi \omega_X^m.
\end{aligned}
\end{equation*}
Its kernel is given by the space $C^\infty_0(X, \bb{R})$ of functions that have fibrewise mean value zero. A key step in the study of optimal symplectic connections is that we can further split this space as follows.

Consider a real vector bundle $E \to B$ \cite[\S 3.1]{DervanSektnan_OSC1}, whose fibre over $b \in B$ is the real finite-dimensional vector space $\mrm{ker}_0(\m{D}_b^*\m{D}_b)$ of holomorphy potentials on the fibre $X_b$ with mean-value zero with respect to $\omega_b$. $E$ is well defined as a vector bundle since we are assuming that he complex dimension of the Lie algebra $\f{h}(X_b)$ of holomorphic vector fields on $X_b$ is independent on $b$.
Its smooth global sections are
\begin{equation*}
C^\infty(E) = \mrm{ker}_0 \m{D}_{\m{V}}^*\m{D}_{\m{V}}.
\end{equation*}
In \cite[Lemma 2.7]{Hallam_geodesics}, Hallam showed - using the Cartan decomposition for the space $\f{h}(X_b)$ of holomorphic vector fields of the fibre - that $E_b$ can be also viewed as the vector space of all K\"ahler potentials $\varphi_b$ on $X_b$ for which $\omega_b + i\del\delbar\varphi_b$ is still cscK.
We can split $C^\infty_0(X)$ as
\begin{equation*}
C^\infty_0(X, \bb{R}) = C^\infty(E) \oplus C^\infty(R)
\end{equation*}
where $C^\infty(R,X, \bb{R})$ is the fibrewise $L^2$-orthogonal complement with respect to the fibre metric $\omega_b$, i.e. for all $\varphi \in \mrm{ker}_0 \m{D}_b^*\m{D}_b$, $\psi \in C^\infty(R)$
\begin{equation*}
\langle \varphi , \psi \rangle_b := \int_{X_b} \varphi \psi \omega_b^m = 0.
\end{equation*}
In the end we obtain
\begin{equation}\label{Eq:splitting_function_space}
C^\infty(X, \bb{R}) = C^\infty(B) \oplus C^\infty(E)\oplus C^\infty(R).
\end{equation}
We denote by $p_E : C^\infty(X) \to C^\infty(E)$ the projection.

Since we are interested in deformations of the complex structure of $X$, sometimes we will denote the vector bundle $E$ as $E(\omega_X, J_0)$. Notice that if we change just the relatively K\"ahler metric $\omega_X$ to $\omega_X + i\del\delbar\varphi$, for $\varphi \in C^\infty(X)$, the vector bundles $E(\omega, J_0)$ and $E(\omega+i\del\delbar\varphi, J_0)$ are isomorphic.
%
%

\subsection{Optimal symplectic connections in the relatively cscK case}\label{Subsec:OSC_classic}
The relative symplectic form $\omega_X$ determines a splitting of the tangent space
\begin{equation}\label{Eq:splitting_tangent_bundle}
TX = \m{V} \oplus \m{H}^{\omega_X},
\end{equation}
where $\m{V}_x = T_x X_{\pi(x)}$ is the tangent space to the fibre, and
\begin{equation*}
\m{H}^{\omega_X}_x = \set*{\left. u \in T_xX\right| \omega_X(u, v)=0 \ \forall v \in \m{V}_x}.
\end{equation*}
In this context $\omega_X$ is called a \emph{symplectic connection} \cite[Chapter 6]{McDuff_Salamon} and it determines the following curvature quantities:
\begin{enumerate}[label=(\roman*)]
\item the \emph{symplectic curvature} is a two-form on $B$ with values in the fibrewise Hamiltonian vector fields defined as, for $v_1, v_2 \in \f{X}(B)$, as
\begin{equation*}
F_{\m{H}}(v_1, v_2) = [v_1^\sharp, v_2^\sharp]^{\mrm{vert}},
\end{equation*}
where $v_j^\sharp$ denotes the horizontal lift. Let $m$ be the map which associates to a fibrewise Hamiltonian vector field its fibrewise Hamiltonian function with fibrewise mean value zero. Thus we consider $m^*(F_{\m{H}})$, which is a two-form on $B$ with values in $C^\infty_0(X, \bb{R})$, and we pull it back on $X$. It is related to the symplectic connection as follows \cite[\S 3.2]{DervanSektnan_OSC1}:
\begin{equation*}
(\omega_X)_{\m{H}} = m^*(F_{\m{H}}) + \pi^*\beta,
\end{equation*}
where $\beta$ is a two-form on $B$.
\item the curvature $\rho$ of the Hermitian connection induced on the top wedge power $\ext{m}\m{V}$. We will consider its purely horizontal part $\rho_{\m{H}}$.
\end{enumerate}
\begin{definition}
\cite[\S 3.3]{DervanSektnan_OSC1}
A relatively cscK metric $\omega_X$ is called an \emph{optimal symplectic connection} if
\begin{equation}\label{Eq:OSC_cscK}
p_E \left(\Delta_{\m{V}}(\Lambda_{\omega_B} m^*(F_{\m{H}}))+ \Lambda_{\omega_B}\rho_{\m{H}} \right) =0.
\end{equation}
This is a second order elliptic equation on the vector bundle $E \to B$. In the following, we will use the notation $\Theta(\omega_X, J) = \Delta_{\m{V}}(\Lambda_{\omega_B} m^*(F_{\m{H}}))+ \Lambda_{\omega_B}\rho_{\m{H}}$.
\end{definition}

The linearisation of the equation at a solution is given by the operator $\m{R}^*\m{R}$ \cite[\S 4.3]{DervanSektnan_OSC1}, where 
\begin{equation}\label{Eq:operator_R}
\m{R}(\varphi_E) = \delbar_B \nabla_{\m{V}}^{1,0} \varphi_E
\end{equation}
and the adjoint is computed with respect to $\omega_F + \omega_B$. Here $\nabla_{\m{V}}^{1,0} \varphi_E$ is a section of the holomorphic tangent bundle; the vertical part of $\delbar \nabla_{\m{V}}^{1,0} \varphi_E$ vanishes since $\varphi \in C^\infty(E)$ and the horizontal part is denoted by the expression \eqref{Eq:operator_R}.
The operator \eqref{Eq:operator_R} can be described as follows \cite[\S 4.3]{DervanSektnan_OSC1}: let $\m{D}_k^*\m{D}_k$ be the Lichnerowicz operator with respect to the K\"ahler metric $\omega_k$. It can be written as a power series expansion in negative powers of $k$:
\begin{equation*}
\m{D}_k^*\m{D}_k = \m{L}_0 + k^{-1}\m{L}_1 + O\left(k^{-2}\right),
\end{equation*}
where $\m{L}_0$ is the \emph{vertical} Lichnerowicz operator $\m{D}_{\m{V}}^*\m{D}_{\m{V}}$. Then for $\varphi, \psi$ fibrewise holomorphy potentials
\begin{equation*}
\int_X \varphi \m{L}_1(\psi) \omega_X^m \wedge \omega_B^n = \int_X \langle \m{R}\varphi, \m{R}\psi \rangle_{\omega_F+\omega_B} \omega_X^m\wedge \omega_B^n.
\end{equation*}
This means that the operator $\m{R}^*\m{R}$ can actually be seen as $p_E \circ \m{L}_1$ restricted to $\m{C}^\infty_E(X)$. The kernel of $\m{R}$, thus of $\m{R}^*\m{R}$, consists of fibrewise holomorphy potentials which are global holomorphy potentials on $X$ with respect to $\omega_k$.

\subsection{Optimal symplectic connections in general}
Let $\pi_Y:(Y, H_Y) \to (B,L)$ be a holomorphic submersion where all the fibres are K-semistable. Let $(\m{X},\m{H}) \to S \times (B, L)$ be a degeneration of the fibration $(Y, H_Y) \to (B, L)$ to a relatively cscK fibration $\pi:(X, H) \to (B, L)$, where $S \subset \bb{C}$.
In particular all the fibrations $\m{X}_s \to B$ are biholomorphic to $Y \to B$ for $s \neq 0$.

Let $J_0$ be a complex structure on $X$ such that $(\omega_X, J_0)$ is relatively cscK. Through the fibrewise analogue of Ehresmann's fibration theorem \ref{Prop:Ehresmann_relative}, we can take the perspective of fixing $\omega_X$ and seeing $\m{X} \to B \times S$ as a family of compatible complex structures  $\{ J_s \}$ which keep $\pi$ a holomorphic submersion and are all biholomorphic except for $J_0$.
Let
\begin{equation*}
\scr{J}_\pi = \set*{J \ \text{almost complex structure compatible with }\omega_X \ \text{and s.t.} \ \dd\pi \circ J = J_B \circ \dd\pi }.
\end{equation*}
Compatibility with $\omega_X$ means that $\omega_X(J\cdot,J \cdot)= \omega_X(\cdot, \cdot)$ and that $\omega_b \circ J|_{X_b}$ is non degenerate and positive definite. The tangent space at $J_0$ to $\scr{J}_\pi$ can be identified with
\begin{equation}\label{Eq:tangent_J_pi}
T_{J_0}^{0,1}\scr{J}_\pi = \set*{\left.A \in \Omega^{0,1}(\m{V}^{1,0}) \right| \omega_F (A\cdot, \cdot) + \omega_F (\cdot, A\cdot) = 0}.
\end{equation}

As in \S \ref{Subsec:Sz_def_theory}, for any fibre $X_b$ let $V_b$ be the Kuranishi space, $K_b$ the group of Hamiltonian isometries and $\Psi_b$ the Kuranishi map \eqref{Eq:Kuranishi_map} of the fibre.
Let $x_{s,b} \in V_b$ be such that $\Psi_b(x_{s,b}) = J_s|_{X_b}$.
Let $\mu_b$ be the moment map \eqref{Eq:Moment_map_Sz}. Then we can define a section of $C^\infty(E)$ as
\begin{equation}\label{Eq:mu_s_section}
\mu_s(b) = p_E \left(S_{X_b}\left(\omega_b, \Psi_b(x_{s,b})\right)\right).
\end{equation}
Note that $\Psi_b$ may not vary smoothly with $b$, but when applied to $x_{s,b}$ it gives the complex structure $J_s|_{X_b}$. Since $J_s$ is a complex structure defined on the whole $X$, it varies smoothly with the base, so $\mu_s$ is a smooth section.
For each fibre $X_b$ we can linearise the action to the tangent space to $V_b$ at $0$ as in \S \ref{Subsec:Sz_def_theory}, so we can define another section $\nu$ of $E$ by
\begin{equation}\label{Eq:nu_section}
\nu(b) = \nu_b(v_b).
\end{equation}
Here $\nu_b$ is the moment map defined in Definition \ref{Def:map_nu} for the linear action of $K_b$ on $\widetilde{H}^1(X_b)$, and $v_b \in \widetilde{H}^1(X_b)$ is tangent at $0 \in V_b$. Even if $\nu_b$ does not necessarily vary smoothly with $b$, $\nu$ is a smooth section because there is an expansion
\begin{equation}\label{Eq:mu_s_expansion}
\mu_s(b) = \frac{s^2}{2}\nu + O(s^3)
\end{equation}
as explained in Proposition \ref{Prop:expansion_mu_nu}, and $\mu$ is smooth.

\begin{definition}
We say that the relative K\"ahler form $\omega_X$ is an \emph{optimal symplectic connection} for the family $(\m{X}, \m{B}) \to B \times S$ if it satisfies the equation
\begin{equation}\label{Eq:genOSC}
p_E \left(\Delta_{\m{V}}(\Lambda_{\omega_B} m^*(F_{\m{H}}))+ \Lambda_{\omega_B}\rho_{\m{H}} \right) + \frac{\lambda}{2} \nu = 0
\end{equation}
for some $\lambda >0$.
\end{definition}
The first term is the left-hand side of the optimal symplectic connection equation \eqref{Eq:OSC_cscK} for fibrewise cscK metrics, and it involves only the complex structure $J_0$. The second term represents the deformation of the complex structure, in terms of the first order deformation of the fibres.
In \S \ref{Subsec:linearisation_OSC}, we will prove that the linearisation of the equation at a solution is given by an operator
\begin{equation*}
\widehat{\m{L}} =\m{R}^*\m{R} + \m{A}^*\m{A},
\end{equation*}
where $\m{R}$ is the operator \eqref{Eq:operator_R} and $\m{A}$ is obtained by extending the map \eqref{Eq:linearised_infinitesimal_action} to a fibrewise map. As shown in Proposition \ref{Prop:kernel_linearisation}, its kernel is given by the fibrewise $J_0$-holomorphy potentials which are global holomorphy potentials with respect to $J_s$.

The definition of an optimal symplectic connection can be generalised as follows.
\begin{definition}\label{Def:extremal_OSC}
$\omega_X$ is an \emph{extremal symplectic connection} on $Y$ if
\begin{equation*}
\widehat{\m{L}}\left(p_E (\Theta(\omega_X, J_0)) + \frac{\lambda}{2}\nu \right) = 0.
\end{equation*}
In particular, the fibrewise $J_0$-holomorphy potential
\begin{equation}\label{Eq:extremal_holomorphy_potential}
h_1 := p_E \left(\Theta(\omega_X, J_0)\right) + \frac{\lambda}{2}\nu
\end{equation}
is a holomorphy potential for the complex structure of $Y$.
\end{definition}

\begin{remark}
In a recent work \cite{SektnanSpotti_Extremal_tc}, Sektnan-Spotti address a similar problem in a specific situation: they prove the existence of extremal metrics in an adiabatic class on the total space of certain test configurations, compactified over $\bb{P}^1$, where the central fibre is cscK and the general fibre is just K-semistable. By using the same $\bb{C}^*$-action it is possible to view such a test configuration as a deformation of a compactified product test configuration with cscK fibre. Their proof does not need the extremal symplectic connection condition but it requires that the rank of the vector bundle $E$ is equal to the dimension of the kernel of the operator $\m{R}$ \eqref{Eq:operator_R}. In particular, this hypothesis implies that $E$ is a trivial vector bundle and it is reasonable to expect that it also implies that the extremal symplectic connection condition \eqref{Def:extremal_OSC} is satisfied (though it seems challenging to actually prove this), thus relating the two constructions.
\end{remark}

\section{Deformations of fibrations}\label{Sec:deformations_of_fibrations}

In this section, we study more in detail the deformations of a holomorphic fibrewise cscK fibration. In particular, we prove a relative version of Ehresmann's fibration theorem in Proposition \ref{Prop:Ehresmann_relative}, which allows us to view families of fibrations as families of complex structures in $\scr{J}_\pi$ on the same underlying smooth fibration.
In \S \ref{Subsec:relative_Kuranishi} we then prove a relative version of Kuranishi's Theorem \ref{Thm:Kuranishi}, which will be needed in Section \ref{Sec:adiabatic_limit} to describe the linearisation of the optimal symplectic connection equation \eqref{Eq:genOSC}.

We start by giving a description in local coordinates of a first-order deformation $A \in T_{J_0}\scr{J}_\pi$.
We fix a local trivialisation of $X \to B$ and we make the following notation conventions:
\begin{enumerate}[label=(\roman*)]
\item[] $\{ w^1, \dots, w^m \}$ are vertical holomorphic coordinates; indices are denoted with the letters $a, b, c, \dots$;
\item[] $\{z^1, \dots, z^n\}$ are holomorphic coordinates on the base; indices are denoted with the letters $i, j, k, \dots$.
\end{enumerate}
We can then write $A \in T_J \scr{J}_\pi$ locally as
\begin{equation*}
A = A^a_{\ \bar{b}} \dd \bar{w}^b \otimes \del_{w^a} + A^a_{\ \bar{j}} \dd\bar{z}^j \otimes \del_{w^a},
\end{equation*}
since $A$ takes values in the vertical vector fields.
The following lemma explains the relation between $A^a_{\ \bar{b}}$ and $A^a_{\ \bar{j}}$.

\begin{lemma}\label{Lemma:deformation_coord_vertical_horizontal}
For $A \in T_{J_0}\scr{J}_\pi$ we have that:
\begin{enumerate}[label =(\roman*)]
\item $A$ vanishes on horizontal vector fields;
\item $A^{a}_{\ \bar{j}} = A^a_{\ \bar{c}}(\omega_{F})^{d\bar{c}} (\omega_X)_{d\bar{j}}$.
\end{enumerate}
\end{lemma}
\begin{proof}
As for the first claim, if $u \in \m{V}$, $v \in \m{H}$, then
\begin{equation}\label{Eq:Av_vertical}
\omega_X(u, Av) = \omega_F(u, Av) = -\omega_F(Au, v) = 0.
\end{equation}
Indeed, the first equality comes from the fact that $Av$ is vertical and $\omega_X$ coincides with $\omega_F$ on a pair of vertical vector fields. The middle equality follows from the compatibility of the deformation with the fibrewise symplectic structure \eqref{Eq:tangent_J_pi}. The last equality follows from the fact that $v$ is horizontal. So $Av$ is horizontal, since the relation \eqref{Eq:Av_vertical} holds for any $u \in \m{V}$; but $Av$ is also vertical. Thus $Av=0$. This proves the first claim.

We prove the second claim. While $\del_{w^a}$, $\del_{\bar{w}^a}$ are vertical vector fields on $X$, $\del_{z^j}$, $\del_{\bar{z}^j}$ are not horizontal in general. So we have a splitting
\begin{equation*}
\del_{\bar{z}^j} = \varepsilon_{\bar{j}} + \eta_{\bar{j}} \qquad \text{with} \qquad \varepsilon_{\bar{j}}\in \m{H}^{\omega_X}, \ \eta_{\bar{j}} \in \m{V}.
\end{equation*}
Then from $\omega_X(\del_{w^a}, \varepsilon_{\bar{j}}) = 0$ it follows that
\begin{equation*}
(\omega_{X})_{a \bar{b}} \ \eta_{\bar{j}}^b = (\omega_{X})_{a \bar{j}}.
\end{equation*}
So $\eta_{\bar{j}} = (\omega_F)^{a\bar{c}}(\omega_{X})_{a\bar{j}}\del_{\bar{w}^c}$. Thus we can write the horizontal part of $\del_{\bar{z}^j}$ as
\begin{equation*}
\varepsilon_{\bar{j}} = \del_{\bar{z}^j} - (\omega_F)^{a\bar{c}}(\omega_{X})_{a\bar{j}}\del_{\bar{w}^c}.
\end{equation*}
Since $A$ takes value in the vertical vector fields, $A(\varepsilon_{\bar{j}}) = 0$, so
\begin{equation*}
0 = -A^{a}_{\ \bar{c}}(\omega_{F})^{d\bar{c}} (\omega_X)_{d\bar{j}} \del_{w^a}+ A^a_{\ \bar{j}} \ \del_{w^a},
\end{equation*}
hence the claim.
\end{proof}

The following lemma explains the relation between a $J \in\scr{J}_\pi$ and the splitting \eqref{Eq:splitting_tangent_bundle} of the tangent bundle of $X$ induced by $\omega_X$, by showing that the elements of $\scr{J}_\pi$ in a neighborhood of $J_0$ differ from $J_0$ only on the vertical vector bundle.
\begin{lemma}\label{Lemma:deformations_preserve_HV}
Any $J \in \scr{J}_\pi$ preserves the splitting of the tangent space $TX = \m{V} \oplus \m{H}^{\omega_X}$. Moreover, $J(u) = J_0(u)$ for all $u \in \m{H}^{\omega_X}$.
\end{lemma}
\begin{proof}
For the first fact to hold, we have to prove that $J(\m{V}) \subseteq \m{V}$ and $J(\m{H}^{\omega_X})\subseteq \m{H}^{\omega_X}$.
\begin{enumerate}[label = (\roman*)]
\item Let $v \in \m{V}$. Then
\begin{equation*}
\dd \pi (Jv) = J_B (\underbrace{\dd \pi (v)}_{=0}) = 0,
\end{equation*}
so $Jv \in \m{V}$.
\item Let $u \in \m{H}^{\omega_X}$. Then $\omega_X(u, v) =0$ for every $v \in \m{V}$. So
\begin{equation*}
\omega_X(Ju, v) = -\omega_X(u, Jv) = 0,
\end{equation*}
since $Jv$ is vertical by the previous step.
\end{enumerate}
To prove that indeed $J(\m{H}^{\omega_X}) = J_0(\m{H}^{\omega_X})$ for all $J \in \scr{J}_\pi$, consider for instance a first order deformation $J_0 + \varepsilon A$. Since $A$ vanishes on horizontal vector fields by Lemma \ref{Lemma:deformation_coord_vertical_horizontal}, if $u \in \m{H}^{\omega_X}$, $(J_0+\varepsilon A)(u) = J_0(u)$. Let now $J_s$ be a path in $\scr{J}_\pi$ which joins $J_0$ to $J$.
Then
\begin{equation*}
\del_s J_s(\m{H}) = A_s (\m{H}) = 0,
\end{equation*}
so for all $s$ $J_s(\m{H}^{\omega_X}) = J_0(\m{H}^{\omega_X})$, from which the claim follows.
\end{proof}
In particular, the last part of the proof shows that the horizontal parts of the operators $\del, \delbar$ with respect to $J_0$ remain the same for any $J$ in $\scr{J}_\pi$.
\begin{remark}
Let $k \gg 0$ be such that $\omega_X + k\omega_B$ is a K\"ahler form on $X$. Then $\scr{J}_\pi \hookrightarrow \scr{J}(\omega_X + k\omega_B)$. Indeed for $J \in \scr{J}_\pi$
\begin{equation*}
\omega_k (J\cdot, J\cdot) = \omega_X(J\cdot, J\cdot) + k \pi^*\omega_B(J\cdot, J\cdot)
\end{equation*}
and $\pi^*\omega_B(J\cdot, J\cdot) = \omega_B(\dd \pi J\cdot, \dd\pi J\cdot) = \omega_B(J_B\dd\pi\cdot, J_B\dd\pi\cdot) = \pi^*\omega_B(\cdot, \cdot)$.
\end{remark}

\subsection{Families of holomorphic submersions}\label{Subsec:family_submersions}
In this section, we give a more rigorous definition of a family of fibrations and we prove a relative version of Ehresmann's fibration theorem. Families of fibrations are in particular families of holomorphic maps, for which a deformation theory has been developed by Horikawa \cite{Horikawa_deformationsI,Horikawa_deformationsII} in the unobstructed case.
\begin{definition}
Let $B$ be a smooth manifold. A family of holomorphic submersions onto $B$ with central fibration $X \to B$ is defined as the data of $(\m{X}, \widehat{\pi}, p, S)$, where:
\begin{enumerate}[label=(\roman*)]
\item $\m{X}$ is a compact complex manifold, and $S$ is a complex manifold;
\item $p : \m{X}\to S$ and $\widehat{\pi} : \m{X}\to B \times S$ are holomorphic submersions and $p = \mrm{pr}_2\circ\widehat{\pi}$;
\item we can pick a distinguished point $0 \in S$ such that $\widehat{\pi}_0$ induces $\pi : X \to B$.
\end{enumerate}

We say that a family of holomorphic maps $(\m{X}, \widehat{\pi}, p, S)$ onto $B$ is \emph{complete} if for any other family $(\m{X}', \widehat{\pi}', p', S')$ with the same central fibre, there exists a map $h : S'\to S$, defined locally on neighbourhoods of the distinguished point, such that the family $(\m{X}', \widehat{\pi}', p', S')$ can be obtained by $(\m{X}, \widehat{\pi}, p, S)$ via pull-back using $h$.
\end{definition}

For a smooth proper morphism $p:M\to N$ with $N$ connected, Ehresmann's fibration theorem \cite[Proposition 6.2.2.]{Huybrechts_ComplexGeometry} says that the fibres of $p$ are all diffeomorphic, and more precisely that $M$ is locally diffeomorphic to a product. We can extend Ehresmann's theorem to our setting as follows.
\begin{proposition}[Relative Ehresmann's theorem]\label{Prop:Ehresmann_relative}
Let $(\m{X}, \widehat{\pi}, p, S)$ be a family of holomorphic maps onto $B$ with $S$ connected. Then there exists a diffeomorphism $\tau: \m{X} \overset{\sim}{\to} X \times S$ which commutes with the projections to $B$, i.e.
\begin{equation*}
\begin{tikzcd}
\m{X} \arrow[r, "\tau", "\sim"']\arrow[d, "\widehat{\pi}"'] & X \times S \arrow[ld, "\pi \times \mrm{id}"]\\
B \times S
\end{tikzcd}
\end{equation*}
\end{proposition}
\begin{proof}
Up to restricting to a segment, we can assume that $S$ is a small open neighborhood of the origin in $\bb{R}$. We can then consider the vector field $u=\del_s$, and view it as a vector field in $B \times S$, denoted by $u'$. This means that, denoting by $\phi^t_{u'}$ its flow and $\mrm{pr}_2: B\times S \to S$ the second projection, $\mrm{pr}_2 \circ \phi^t_{u'} = \phi^t_{u} \circ \mrm{pr}_2$, i.e. $u'$ is $\mrm{pr}_2$-related to $u$.
It is a consequence of the implicit function theorem that if $F : M \to N$ is a smooth \emph{submersion} of manifolds, then for any vector field on $N$ there exists a vector field on $M$ which is $F$-related to it.
Let then $v$ be a vector field on $\m{X}$ which is $\widehat{\pi}$-related to $u'$, i.e.
\begin{equation*}
\widehat{\pi} \circ \phi^t_v = \phi^t_{u'} \circ \widehat{\pi}.
\end{equation*}

Then, using $p = \mrm{pr}_2\circ \widehat{\pi}$, we obtain
\begin{equation*}
p \circ \phi^t_v = \phi^t_{u} \circ p,
\end{equation*}
thus $v$ is $p$-related to $u$. Hence we can define a map
\begin{equation*}
\begin{aligned}
\tau:\m{X} &\longrightarrow X \times S \\
z &\longmapsto (\phi_v^{-t}(z), p(z)),
\end{aligned}
\end{equation*}
which is a diffeomorphism with inverse
\begin{equation*}
(x,s) \longmapsto \phi^s_v (x).
\end{equation*}
Since $v$ is $\widehat{\pi}$-related to $u'$, this diffeomorphism commutes with the projections to $B$, as required.
\end{proof}

Let now $(X, H) \to (B, L)$ be a polarised holomorphic submersion with $\omega_X \in c_1(H)$ a symplectic form which restricts to a K\"ahler metric on the fibres. Then we consider the following setting:
\begin{enumerate}[label = (\roman*)]
\item $(\m{X}, \m{H}, \widehat{\pi}, p, S)$ is a smooth \emph{polarised family} of maps onto $B$ with central fibration $(X, H) \to B$, where we assume for simplicity that $S$ is a disk in $\bb{C}$. In particular the line bundle $\m{H}$ on $\m{X}$ restricts to a relatively ample line bundle $\m{H}_s$ on each fibration $\m{X}_s \to B$;
\item We assume to have an action of $\bb{C}^*$ on $S \times B$ which lifts to $(\m{X},\m{H})$ such that $\widehat{\pi}$ is $\bb{C}^*$-equivariant, so for $s \neq 0$ the fibrations $(\m{X}_s, \m{H}_s) \to B$ are all biholomorphic. 
\end{enumerate}

Ehresmann's theorem implies that we can locally trivialise the family in such a way that all $\m{X}_s$ are diffeomorphic, so we can interpret the family as a family of almost complex structures $\{J_s\}$ on $X$ which preserve the projection onto $B$.

Moreover, since $(\m{X}_s, H_s)$ is a small deformation of $(X, H)$, we have that $c_1(H_Y) = c_1(H)$. Then by Moser's theorem \cite[Theorem 7.2]{DaSilva_Lectures_symplectic_geometry} we can assume that $\omega_X$ is relatively K\"ahler with respect to the complex structures $J_s$, up to modifying $(\m{X}_s, H_s)$ by a small diffeomorphism.
So we can view a family $\m{X} \to B \times S$ as a family of complex structures on $X$ which keep $\pi$ a holomorphic submersion and $\omega_X$ a relatively K\"ahler metric.


\subsection{Relative Kuranishi's Theorem}\label{Subsec:relative_Kuranishi}

As in the previous section, we consider a holomorphic submersion $\pi:(X, H)\to B$ with a relative K\"ahler metric $\omega_X$ and a complex structure $J_0$ which is fibrewise cscK. We require a relative version of Sz\'ekelyhidi's and Br\"onnle's deformation theory described in \S \ref{Subsec:Sz_def_theory}.
Consider the map
\begin{equation*}
\begin{aligned}
P_{\m{V}}:C^\infty_0(X, \bb{R}) &\longrightarrow T_{J_0}^{0,1}\scr{J}_\pi \\
\varphi &\longmapsto \delbar_{\m{V}} (\mrm{grad}^{\omega_F}\varphi)^{1,0},
\end{aligned}
\end{equation*}
which is the relative version of the map defined in \ref{Def:map_P}.
Let $\widetilde{H}^1_{\m{V}}$ be the kernel in $T^{0,1}_{J_0}\scr{J}_\pi$ of the operator
\begin{equation*}
\square_{\m{V}} = P_{\m{V}}P_{\m{V}}^* + (\delbar^*\delbar)^2
\end{equation*}
inside $T_{J_0}^{0,1}\scr{J}_\pi$.  This is an elliptic operator because $P_{\m{V}}P_{\m{V}}^*$ is trivial in horizontal directions, where the adjoint is computed with respect to any K\"ahler metric on $X$ which restricts to $\omega_F$ vertically. So its kernel is a finite dimensional vector space and it can be described as
\begin{equation*}
\widetilde{H}^1_{\m{V}} = \set*{\left. \alpha \in T^{0,1}_{J_0}\scr{J}_\pi \right| P_{\m{V}}^*\alpha = 0 = \delbar \alpha}.
\end{equation*}
Fibrewise, $\square_{\m{V}}$ restricts to the operator \eqref{Eq:box_laplacian} and $\widetilde{H}^1_{\m{V}}$ restricts to the vector space described in \eqref{Eq:H_tilde_cscK}. In particular $\widetilde{H}^1_{\m{V}}$ depends only on the vertical part of the metric, $\omega_F$.

Consider the smooth fibre bundle $\m{K} \to B$, where $K_b = \mrm{Isom}(X_b, \omega_b)$ is the stabiliser of $\left. J_0 \right|_{X_b}$ for the $\scr{G}_b$-action. Notice that, thanks to our hypothesis, the groups $K_b$ are finite dimensional with dimension independent of $b$.
The group of global sections of $\m{K}$ is
\begin{equation}\label{Eq:rel_K}
K := \mrm{Isom}(\pi, \omega_X) = \set*{f \in \mrm{Aut}(X)| f^*\omega_X = \omega_X \ \text{and} \ \pi \circ f = \pi}.
\end{equation}

We next prove a relative version of Kuranishi's Theorem, adapted from Chen-Sun \cite[\S 6]{ChenSun_CalabiFlow}.

\begin{theorem}[Relative Kuranishi's Theorem]\label{Thm:relative_Kuranishi}
There exists a neighborhood of the origin $V \subset \widetilde{H}^1_{\m{V}}$ and a $K$-equivariant holomorphic map
\begin{equation*}
\Psi : V \rightarrow \scr{J}_\pi
\end{equation*}
such that:
\begin{enumerate}[label = \arabic*)]
\item $\Psi(0) = J_0$;
\item If $v_1, v_2 \in V$ and $\left.v_1\right|_b \in K_b^{\bb{C}}\cdot \left. v_2\right|_b$ for all $b$, and if $\Psi (v_1)$ is integrable, then $\left. \Psi(v_1) \right|_{X_b}$ is in the same $\scr{G}_b^c$-orbit as $\left. \Psi(v_2) \right|_{X_b}$;
\item For any $J \in \scr{J}_\pi$ integrable close to $J_0$, there exists $J'$ in the image of $\Psi$ such that, for all $b$, $J_b'$ is in the same $\scr{G}_b^c$-orbit as $J_b$.
\end{enumerate}
\end{theorem}
\begin{proof}
We can identify any $J$ close to $J_0$ with an element $\alpha \in T_{J_0}^{0,1}\scr{J}_\pi$, i.e. with a $(0,1)$-form with values in the vertical holomorphic tangent bundle, plus the compatibility condition with $\omega_F$. So we have an embedding from an open subset in $T_{J_0}^{0,1}\scr{J}_\pi$ into $\scr{J}_\pi$:
\begin{equation*}
f : \m{U}(T_{J_0}^{0,1}\scr{J}_\pi )\hookrightarrow \scr{J}_\pi.
\end{equation*}

Given $b \in B$, we denote by $\rho_b$ the restriction $\scr{J}_\pi \to \scr{J}(X_b)$. Then $f_b(\alpha|_{X_b})= \rho_b \circ f(\alpha)$. We define now a new embedding $\hat{f}: \m{U}(T_{J_0}^{0,1}\scr{J}_\pi) \hookrightarrow \scr{J}_\pi$ as follows:
\begin{equation*}
\hat{f} (\alpha) = \int_{\m{K}/B} g^{-1} f(g \cdot \alpha) \dd \mu_{\m{K}/B} (g) \quad\text{i.e.}\quad\hat{f} (\alpha)|_{X_b} = \int_{K_b} g|_b^{-1} f(g|_b \cdot \alpha|_{X_b}) \dd \mu_{X_b} (g|_b),
\end{equation*}
where $\dd \mu_{\m{K}/B}$ is the fibrewise Haar measure on $\m{K} \to B$. Then $\hat{f}$ is such that
\begin{equation*}
\hat{f}(k \cdot \alpha)|_{b} = k|_{b} \cdot \hat{f}(\alpha) |_{b}.
\end{equation*}
Now, $J$ is integrable if and only if the corresponding $\alpha$ satisfies \cite[Lemma 6.1.2]{Huybrechts_ComplexGeometry}
\begin{equation}\label{Eq:Maurer-Cartan}
N(\alpha) = \delbar \alpha + [\alpha, \alpha] = 0.
\end{equation}
Note that if $J$ is integrable its restriction to each fibre is also integrable, so equation \eqref{Eq:Maurer-Cartan} holds also fibrewise.
For any $b \in B$, let $H_b : T_{J_0}^{0,1}\scr{J}_\pi|_{X_b} \to \tilde{H}^1_b$ be the $L^2_k$-orthogonal projection and let $G_b$ be the Green operator of $\square_b$:
\begin{equation*}
\idmatrix = G_b \square_b + H_b = \square_b G_b + H_b. 
\end{equation*}
For $\alpha$ integrable, a simple computation starting from \eqref{Eq:Maurer-Cartan} leads to the identity
\begin{equation*}
\alpha|_b + G_b\delbar_b^*\delbar_b\delbar_{b}^*[\alpha|_b, \alpha|_b] = H_b\alpha|_b.
\end{equation*}
Then we can define a map
\begin{equation*}
\begin{aligned}
F : B \times T_{J_0}^{0,1}\scr{J}_\pi &\to T_{J_0}^{0,1}\scr{J}_\pi \\
(b,\alpha) &\mapsto \alpha|_b + G_b\delbar_{b}^*\delbar_{b}\delbar_{b}^*[\alpha|_b, \alpha|_b],
\end{aligned}
\end{equation*}
where both spaces are endowed with the Sobolev $L^2_k$-norm. The differential of $F$ in the second component at the origin is the identity, since $G_b\delbar^*_b\delbar_b\delbar^*_b[\alpha|_b, \alpha|_b]$ is quadratic in $\alpha$.
Hence by the implicit function theorem we can locally invert $F$ and the inverse varies smoothly with $b$. We consider the inverse restricted to an open ball in $\tilde{H}^1_{\m{V}}$, which we define to be $V$, and for $x\in V$ we denote it by $\alpha(x)$.
Thus we have a family
\begin{equation}\label{Eq:family_construction}
U:=\set{\alpha(x) | x \in V} \subset T_{J_0}^{0,1}\scr{J}_\pi,
\end{equation}
and we can define
\begin{equation*}
\begin{aligned}
\Psi : V &\to \scr{J}_\pi \\
x &\mapsto \hat{f}(\alpha(x)).
\end{aligned}
\end{equation*}

We begin by proving that this map satisfies the required properties. Denoting
\begin{equation*}
U^{\mrm{int}} = \set*{\alpha(x)| N(\alpha(x)) = 0}
\end{equation*}
and
\begin{equation*}
U^{\mrm{int}}_{\m{V}} = \set*{\alpha(x) | N_b(\alpha(x)|_b) = 0 \ \forall b \in B},
\end{equation*}
we want to prove that $\mrm{U}^\mrm{int}$ is an analytic subset of $U$. We begin by showing that $U^{\mrm{int}}_{\m{V}}$ is an analytic subset of $U$.
On each fibre $X_b$, $\alpha(x)|_b$ is integrable if and only if $H_b[\alpha(x)|_b, \alpha(x)|_b] = 0$. Indeed
\begin{equation}\label{Eq:N(alpha(x))}
\begin{aligned}
N_b\left(\alpha(x)|_b\right) &= \delbar_b \alpha(x)|_b + \left[\alpha(x)|_b, \alpha(x)|_b\right]\\
&= 2G_b\delbar^*_b\delbar_b\delbar^*_b\left[N(\alpha(x)|_b), \alpha(x)|_b\right] + H\left[\alpha(x)|_b, \alpha(x)|_b\right].
\end{aligned}
\end{equation}
The map
\begin{equation*}
\begin{aligned}
B \times V &\to \widetilde{H}_{\m{V}}\\
(b,v) &\mapsto H_b[\alpha(x)|_b, \alpha(x)|_b]
\end{aligned}
\end{equation*}
is holomorphic, so $U^{\mrm{int}}_{\m{V}}$ is an analytic subset of $U$.
Then, denoting $\overline{U}^{\mrm{int}}$ the analytic family given by the Kuranishi Theorem \cite[Lemma 6.1]{ChenSun_CalabiFlow} applied to $X$, we see that $U^{\mrm{int}}$ is the intersection of $\overline{U}^{\mrm{int}}$ and $U^{\mrm{int}}_{\m{V}}$, so it is itself an analytic family.
Moreover, when restricted to the fibre $X_b$, both maps $f$ and $F$ are $K_b$-equivariant and holomorphic, so (2) is also proved.

We prove (3). Let $J_x = \Psi(x) \in \scr{J}_\pi$, and fix $b \in B$. Given $\xi \in \Gamma(X, \m{V})$ a vertical vector field, define
\begin{equation*}
\begin{aligned}
F_\xi : X &\to X \\
p &\mapsto \mrm{exp}_p(\xi_p, g_{\pi(p)}),
\end{aligned}
\end{equation*}
where $g_{\pi(p)}$ is the Riemannian metric on the fibre $X_{\pi(p)}$ with respect to $J_x$.
Following \cite[Lemma 6.1]{ChenSun_CalabiFlow}, we fix $v \in V_\sigma \subset V$, thought of as a tangent vector at $x$ in $V$, and we define a map
\begin{equation*}
\begin{aligned}
R_{b} : \m{U}(L^2_{k+2}(X_b, \bb{C})) &\to L^2_{k+2}(X_b, TX_b) \\
\varphi_b &\mapsto \xi_b(\varphi_b, v|_b)
\end{aligned}
\end{equation*}
such that $R_{b}(0) = 0$ and
\begin{enumerate}[label = (\roman*)]
\item $\dd_0 R_{b} (\phi_b) = \mrm{grad}^{\omega_b} (\mrm{Re}(\phi_b)) + J_v|_{X_b} \mrm{grad}^{\omega_b}(\mrm{Im}(\phi_b))$;
\item \label{Item:preserves_G^c-orbits}$F_{{\xi}_b(\varphi_b, v|_b)}^*J_x |_{X_b} \in \scr{G}^c_b \cdot J_x|_{X_b}$.
\end{enumerate}

Since this map is defined via the implicit function theorem, the vector field $\xi_b(\varphi_b,v|_b)$ varies smoothly with $b$, thus defining a global vertical vector field on $X$ (more details about this technique of using the implicit function theorem to prove smooth dependence on $b$ are given in the proof of Proposition \ref{Prop:Kuranishi_map_cscK} below).
So we can define a global map
\begin{equation*}
\begin{aligned}
R : \m{U}(L^2_{k+2}(X, \bb{C})) &\to L^2_{k+2}(X, \m{V}) \\
\varphi &\mapsto Y(\varphi, v) \quad \text{s.t.} \quad \xi(\varphi, v)|_{X_b} = \xi(\varphi|_b, v|_b).
\end{aligned}
\end{equation*}
The complex structure $F_{Y(\varphi, v)}^*J_x$ on $X$ satisfies the following properties:
\begin{enumerate}
\item is compatible with $\omega_X$. Indeed
\begin{equation*}
\begin{aligned}
&\omega_X(F_{Y(\varphi, v)}^*J_x \cdot, \cdot) + \omega_X(\cdot, F_{Y(\varphi, v)}^*J_x\cdot) =  \\
&\omega_F(F_{Y(\varphi, v)}^*J_x \cdot, \cdot) + \omega_F(\cdot, F_{Y(\varphi, v)}^*J_x\cdot) + \omega_{X, \m{H}}(J_x\cdot, \cdot) + \omega_{X, \m{H}}(\cdot, J_x\cdot).
\end{aligned}
\end{equation*}
The first two terms sum to zero because the complex structure $F_{\xi(\varphi, v)}^*J_x$ is fibrewise compatible with the fibrewise K\"ahler form. The last two terms sum to zero since $J_x \in \scr{J}_\pi$.
\item preserves $\pi$, since the differential commutes with pull-back.
\item satisfies property \ref{Item:preserves_G^c-orbits} above for every $b \in B$.
\end{enumerate}
Let now $\alpha(\varphi, v)$ be the $(0,1)$-form with values in the holomorphic vertical tangent space which is the pre-image of $F_{\xi(\varphi, v)}^*J_x$ via $\Psi$. Then from \eqref{Eq:N(alpha(x))} it follows that $\alpha(\varphi, v)$ fibrewise satisfies an elliptic equation of the form
\begin{equation}\label{Eq:elliptic_alpha(phi,v)}
\square_{\m{V}} T(\varphi, v, N(\alpha)) = 2\delbar^*_{\m{V}}\delbar_{\m{V}}\delbar^*_{\m{V}}\left[T(\varphi, v, N(\alpha)), S(\varphi, v, \alpha)\right],
\end{equation}
where $T(0, v, N) = N$ and $S(0, v, \alpha) = \alpha$.

Let now $J \in \scr{J}_\pi^{\mrm{int}}$ be close to $J_0$ in $L^2_k$. The proof now goes exactly as in \cite[Lemma 6.1]{ChenSun_CalabiFlow}, and we report it here for completeness. Since $J$ is integrable, the corresponding vector-valued $(0,1)_{J_0}$-form $\alpha_J$ satisfies \eqref{Eq:elliptic_alpha(phi,v)} for all $(\varphi, v)$. Consider the $L^2_k$-projections
\begin{equation*}
\Pi_1 : L^2_k\left(T_{J_0}^{0,1}\scr{J}_\pi\right) \to \mrm{Im}(P_{\m{V}}), \qquad \Pi_2: L^2_k\left(T_{J_0}^{0,1}\scr{J}_\pi\right) \to \widetilde{H}^1_{\m{V}}
\end{equation*}
and consider the map $\chi : \m{U}(L^2_{k+2}(X, \bb{C})) \times V_\sigma \to \mrm{Im}(P_{\m{V}})\times \widetilde{H}^1_{\m{V}}$ defined by
\begin{equation*}
(\varphi, v) \mapsto \left(\Pi_1\left(F_{Y(\varphi, v)}^*J_v\right), \Pi_2\left(F_{Y(\varphi, v)}^*J_v\right)\right).
\end{equation*}

Remark that if $\alpha, \beta \in T_{J_0}^{0,1}\scr{J}_\pi$ satisfy \eqref{Eq:elliptic_alpha(phi,v)} and they are such that $(\Pi_1\alpha, \Pi_2\alpha) = (\Pi_1\beta, \Pi_2\beta)$, then by ellipticity it follows that $\alpha = \beta$.
The differential of the map $\chi$ is $\dd_{(0,0)}\chi (\varphi, v) = (P_{\m{V}}(\varphi), v)$: it is surjective and the kernel corresponds to fibrewise holomorphy potentials, so it is finite dimensional fibrewise. Thus again by the implicit function theorem, there exist $(\varphi, v)$ such that $(\Pi_1(\alpha_J), \Pi_2(\alpha_J)) = \chi(\varphi, v)$, hence by the ellipticity argument $\alpha_J = F_{Y(\varphi, v)}^*J_x$.
\end{proof}

\begin{remark}[Versal deformations]
The proof of the relative Kuranishi Theorem guarantees the existence of \emph{versal} deformations, i.e. complete and effective. More precisely, a deformation $\m{X} \to B \times S$ with central fibre $(X, J_0)$ is called \emph{versal} if any other family $\m{X}' \to B \times S'$ (centred at $J_0$) is obtained by pullback via a map $f : S' \to S$, which might not be unique but whose differential is uniquely determined.
This is proven in the third step of Theorem \ref{Thm:relative_Kuranishi}, where a single complex structure $J$ is considered instead of a second family $\{ J_{s'} \}$. The pullback is given by the exponential map $F_{\xi}$, where $\xi = \xi(\varphi, v)$ is uniquely determined by the vector $v$, which is tangent to the complex structure $J$.
\end{remark}

\begin{proposition}\label{Prop:Kuranishi_map_cscK}
Possibly after shrinking $V$, we can perturb the map $\Psi$ to
\begin{equation}\label{Eq:rel_Kuranishi_map_deformed}
\Phi : V \to \scr{J}_\pi
\end{equation}
such that:
\begin{equation*}
S_{\m{V}}\left(\omega_X, \Phi(x)\right) \in C^\infty(E, J_0)
\end{equation*}
\end{proposition}
Remark that the claim holds fibrewise as a consequence of Theorem \ref{Thm:Kuranishi}, so we just need to check that the complex structure we find on each fibre $X_b$ varies smoothly with $b$. This relies on the fact that the proof involves the implicit function theorem.

\begin{proof}[Proof of Proposition \ref{Prop:Kuranishi_map_cscK}]
For every $b \in B$, the Lie algebra of $K_b$ is $\f{k}_b = \ker \m{D}_b^*\m{D}_b$, which is exactly the fibre $E_b$ of the vector bundle $E$ defined in \S \ref{Subsec:splitting_function_space}. Let $U_l$ be a small ball around the origin of $L^{2,l}(R)$. Define a map
\begin{equation*}
\begin{aligned}
G : B \times V \times U_l &\to L^{2,l-4}(R) \\
(b,x,\varphi) &\mapsto \pi_{L^{2,l-4}(R)} S \left(\omega_b, F_{\varphi_b}(\Psi(x)|_{b})\right),
\end{aligned}
\end{equation*}
where $\Psi(x)|_b$ defines an element in $\scr{J}_\pi$ from Theorem \ref{Thm:relative_Kuranishi}, and the map $F$ is the one defined in \eqref{Eq:map_F}.
The derivative along the third component of $G$ at 0 of a function $\varphi$ is given by $P_{\m{V}}^*P_{\m{V}}(\varphi)$, which is an isomorphism $L^{2,l}(R) \to L^{2,l-4}(R)$.
By the implicit function theorem, for every $b$, $\Psi$ can be perturbed to $\Phi_b : V_b \to \scr{J}(X_b)$ in such a way that $S(\omega_b, \Phi_b(x)) \in \f{k}_b$, and $\Phi_b$ varies smoothly with $b$. Thus we find a map
\begin{equation*}
\begin{aligned}
\Phi : V &\to \scr{J}_\pi
\end{aligned}
\end{equation*}
such that $S_{\m{V}}(\omega_X, \Phi(x)) \in C^\infty(E)$.
\end{proof}

Let us return now to considering a holomorphic submersion $\pi: X \to B$ with a relatively cscK metric $(\omega_X, J_0)$. By viewing $\omega_X$ as fixed and varying the complex structure, we consider a family $\{J_s\}$ such that $(\omega_X, J_s)$ are relatively K\"ahler metrics on $X \to B$.
Theorem \ref{Thm:relative_Kuranishi} allows us to extend definition of the sections $\mu_s$ \eqref{Eq:mu_s_section} and $\nu$ \eqref{Eq:nu_section} of $C^\infty(E)$ to the following maps.
\begin{definition}\label{Def:mu_pi_nu_pi}
Define the maps
\begin{equation*}
\begin{aligned}
\mu_\pi : V &\to C^{\infty}(E, J_0)\\
x &\mapsto \mrm{Scal}_{\m{V}}(\omega_X, \Phi(x))
\end{aligned}
\end{equation*}
and
\begin{equation*}
\begin{aligned}
\nu_\pi : \widetilde{H}^1 &\to C^\infty(E) \\
v &\mapsto \nu_\pi(v) \qquad
\end{aligned}
\end{equation*}
where $\nu_\pi(v)|_b = \nu_b(v|_{X_b})$, for $\nu_b$ defined in \ref{Def:map_nu}.
\end{definition}
\begin{remark}
If $x_s\in V$ corresponds to $J_s$ via the relative Kuranishi map \eqref{Eq:rel_Kuranishi_map_deformed}, we have that $\mu_\pi(x_s) = \mu_s$, where $\mu_s \in C^\infty(E)$ is the section defined in \eqref{Eq:mu_s_section}. Similarly, if $v \in \widetilde{H}^1_{\m{V}}$ is the deformation of the family $\{J_s\}$, then $\nu_\pi(v)$ is the section $\nu \in C^\infty(E)$ defined in \eqref{Eq:nu_section}. By applying Proposition \ref{Prop:Kuranishi_map_cscK} we can perturb $\mu_\pi$ to end up in $C^\infty(E)$, so we do not see the projection as in \eqref{Eq:mu_s_section}. 
\end{remark}

From the definition \eqref{Eq:mu_s_section}, the perturbation given by Proposition \ref{Prop:Kuranishi_map_cscK} and the expansion \eqref{Eq:mu_s_expansion} of $\mu_s \in C^\infty(E)$ it follows that, if $v \in \widetilde{H}^1_{\m{V}}$ is the deformation of the family $\{J_s\}$,
\begin{equation}\label{Eq:mu_pi_expansion}
\mrm{Scal}_{\m{V}}(\omega_X, J_s) = \mu_\pi(x_s) = \widehat{S}_b + \frac{s^2}{2} \nu(v) + O(s^3).
\end{equation}

\section{Extremal metrics on the total space}\label{Sec:adiabatic_limit}
As before, let $\widehat{\pi}:(\m{X},\m{H}) \to (B,L) \times S$ be a degeneration of a fibration $\pi_Y:(Y, H_Y) \to B$ with central fibration $\pi:(X, H)\to B$, endowed with a $\bb{C}^*$-action on $B \times S$ which lifts to $(\m{X}, \m{H})$. Let $\omega_X$ be a relatively cscK metric on $X$; from the discussion at the end of Section \ref{Subsec:family_submersions} we can assume that $\omega_X$ is relatively K\"ahler also on $Y$. It follows that the general fibrations $X_s \to B \times \{ s\}$ are all biholomorphic to $Y \to B$.
For $k \gg 0$, consider the K\"ahler form
\begin{equation*}
\omega_k = \omega_X + k \pi^*\omega_B,
\end{equation*}
where $\omega_B$ is a fixed K\"ahler metric on $B$. We will omit the pull-back in the notation, and write just $\omega_X + k \omega_B$.

We will later need to choose $\omega_B$ appropriately, to produce cscK and extremal metrics on $Y$. To do so, we need the following definitions from the moduli theory of cscK manifolds \cite{Fujiki_Schumacher_Moduli_cscK,Inoue_ModuliSpaceFano,DervanNaumann_ModuliCscK}. In \cite{DervanNaumann_ModuliCscK} it is shown that there exists a complex space $\m{M}$ which is the moduli space of polarised cscK manifolds and that there exists a K\"ahler metric on $\m{M}$ called the \emph{Weil-Petersson metric}. Moreover, since our central fibration $\pi: X \to B$ has cscK fibres, there is an induced map $q:B \to \m{M}$. The pull-back via $q$ of the Weil-Petersson metric, denoted $\alpha_{WP}$, is a closed smooth $(1,1)$-form on $B$, and it has the following expression \cite{DervanNaumann_ModuliCscK, Fujiki_Schumacher_Moduli_cscK}:
\begin{equation}\label{Eq:WPform}
\alpha_{WP} = \frac{\widehat{S}_b}{m+1} \int_{X/B} \omega_X^{m+1} - \int_{X/B} \rho\wedge\omega_X^m,
\end{equation}
where $\rho$ is the relative Ricci form defined in \S \ref{Subsec:OSC_classic} and $m$ is the dimension of the fibres. Notice that $\alpha_{WP}$ is positive semi-definite in general.

\begin{definition}\label{Def:twisted_cscK_extremal}
We say that $\omega_B \in c_1(L)$ is
\begin{enumerate}
\item \emph{twisted cscK} with respect to $\alpha$ if $\mrm{Scal}(\omega_B) - \Lambda_{\omega_B} \alpha = c_B$;
\item \emph{twisted extremal} with respect to $\alpha$ if $\mrm{Scal}(\omega_B) - \Lambda_{\omega_B} \alpha \in \ker \m{D}_B$, where $\m{D}_B$ is the Lichnerowicz operator on $B$.
\end{enumerate}
\end{definition}
\begin{definition}\label{Def:aut(q)}
We define the group of automorphisms of the moduli map to be
\begin{equation*}
\mrm{Aut}(q) = \set*{f \in \mrm{Aut}(B,L) | q \circ f = f}.
\end{equation*}
\end{definition}
In \cite[\S 3.2]{DervanSektnan_ExtremalFibrations} it is shown that, denoted $h_B$ the twisted extremal holomorphy potential, the linearisation of the twisted extremal operator at a solution is given by the map
\begin{equation}\label{Eq:linearisation_twisted_eq}
\m{L}_{\alpha}(\varphi) = -\m{D}_B^*\m{D}_B \varphi + \frac{1}{2}\langle \nabla \Lambda_{\omega_B}\alpha, \nabla\varphi\rangle + \langle i\del\delbar\varphi, \alpha\rangle.
\end{equation}
Moreover, in \cite[Proposition 3.5]{DervanSektnan_ExtremalFibrations} it is proven that kernel of this operator is given by the holomorphy potentials of those vector fields whose flow lies in $\mrm{Aut}(q)$.

In what follows will also need the groups of automorphisms of the projections $\pi$ and $\pi_s$.
\begin{definition}\label{Def:aut(pi)}
For $\pi:X \to B$ we define
\begin{equation*}
\mrm{Aut}(\pi) = \set*{f \in \mrm{Aut}(X, H) | \pi \circ f = \pi},
\end{equation*}
and similarly for $\pi_s : X_s \to B$.
\end{definition}

\subsection{Expansion of the scalar curvature}
In this subsection, we derive an expansion of the scalar curvature $\mrm{Scal}(\omega_k, J_s)$, in powers of $s$ and inverse powers of $k$, from which we deduce the optimal symplectic connection equation \eqref{Eq:genOSC}.
Recall from \cite[\S 4.1]{DervanSektnan_OSC1} that
\begin{equation*}
\begin{aligned}
\mrm{Scal}(\omega_k, J_s) = \mrm{Scal}_\m{V}(\omega_X, J_s) + k^{-1}\left( \mrm{Scal}(\omega_B) + \Updelta_{\m{V}} (\Lambda_{\omega_B} (\omega_X)_{\m{H}}) + \Lambda_{\omega_B} \rho_{\m{H}}\right) + O\left(k^{-2}\right).
\end{aligned}
\end{equation*}
Clearly, the $k^{-1}$ term - denoted $T_{k^{-1}}$ - depends on $s$, so we can write
\begin{equation*}
T_{k^{-1}}(\omega_B, \omega_X, J_s) = T_{k^{-1}}(\omega_B, \omega_X, J_0) + O(s) .
\end{equation*}
\begin{proposition}
By choosing $s^2 = \lambda k^{-1}$ for $\lambda >0$ and using the expansion \eqref{Eq:mu_pi_expansion} for the vertical scalar curvature we obtain
\begin{equation*}
\mrm{Scal}(\omega_k, J_s) = \widehat{S}_b + k^{-1}\left( \psi_B + p_E(\Updelta_{\m{V}} (\Lambda_{\omega_B} (\omega_X)_{\m{H}}) + \Lambda_{\omega_B} \rho_{\m{H}}) + \frac{\lambda}{2} \nu(v) + \psi_R\right) + O\left(k^{-3/2}\right)
\end{equation*}
where:
\begin{enumerate}[label=(\roman*)]
\item $\psi_B$ is a function on the base given by
\begin{equation*}
\psi_B = \mrm{Scal}(\omega_B) + \int_{X/B} \left(\Lambda_{\omega_B}\rho_{\m{H}}\right) \omega_X^m.
\end{equation*}
\item $\psi_R \in C^{\infty}(R, J_0)$.
\end{enumerate}
\end{proposition}
\begin{proof}
For the first item, in \cite[\S 4.1]{DervanSektnan_OSC1} based on \cite[Lemma 2.3]{Fine_fibrations_CMlinebundle} it is shown that
\begin{equation*}
\int_{X/B} (\Lambda_{\omega_B}\rho_{\m{H}}) \omega_X^m = - \Lambda_{\omega_B} \alpha_{\msc{WP}},
\end{equation*}
where $\alpha_{\msc{WP}}$ is the Weil-Petersson metric defined in \eqref{Eq:WPform}.
Moreover, the $k^{-1}$-term depends only on $J_0$ because the $O(s)$-part ends up in $O\left(k^{-3/2}\right)$. Its expression is obtained following \cite[\S 4.2]{DervanSektnan_OSC1}.
\end{proof}

Thus the optimal symplectic connection equation implies that the $C^\infty(E)$-part of the $k^{-1}$-term of the expansion of the scalar curvature vanishes. Note that $\widehat{S}_b$ is a topological constant independent on $b$, since all the fibres are diffeomorphic.

\subsection{Linearisation of the fibrewise map $\nu$}
We restrict our attention to a single fibre, so we consider a manifold $(M, \omega)$, where $J_0$ is a cscK complex structure and $v\in \widetilde{H}^1$ is a deformation of $J_0$. We wish to linearise the map $\nu$ defined in \ref{Def:map_nu}.

Let $\varphi_E \in \mrm{ker}_{\bb{R}}(\m{D}_0^*\m{D}_0)$. Then $\frac{1}{2}\nabla^{g}\varphi_E$ is a real holomorphic vector field, where $g$ is the Riemannian metric induced by $\omega$ and $J_0$.
Let $\rho(t)$ be the flow of the vector field
\begin{equation*}
Y_{\varphi_E} := \nabla^{g}\varphi_E = - J_0\mrm{grad}^{\omega}\varphi_E.
\end{equation*}
We wish to study how $\nu(v)$ changes when changing $\omega$ to $\rho(t)^*\omega$, so we must compute
\begin{equation}\label{Eq:der_nu_t_v_t}
\del_t|_{t=0} \nu_t(v_t) = \del_t|_{t=0}\nu_t(v) + \dd_v\nu(\del_t|_{t=0} v_t).
\end{equation}
In this expression, $\nu_t$ is the moment map for the action of $K^{\bb{C}}$ defined in \ref{Def:map_nu} computed with respect to the K\"ahler form $\rho(t)^*\omega$, and $v_t = \rho(t)^*v$.

Remark that $\rho(t)$ is a 1-parameter group of diffeomorphisms in $K^{\bb{C}}$ because it is the flow of a holomorphic vector field that admits a holomorphy potential.

As in \S \ref{Subsec:Sz_def_theory}, let $\Phi : V \to \scr{J}$ be the Kuranishi map \eqref{Eq:Kuranishi_map_deformed} which maps $0$ to $J_0$, and $\widetilde{H}^1$ the deformation space, which we identify with the tangent space $T_0V$. The map $\Phi$ is $K$-equivariant, hence locally $K^{\bb{C}}$-equivariant. In particular
\begin{equation*}
\Phi\left(\rho(t)^*x\right) = \rho(t)^*\Phi(x) \qquad \text{for} \ x\in V.
\end{equation*}
Hence the pair $(\omega, \rho(t)^*x)$ corresponds via $\Phi$ to a compatible pair $(\omega, \rho(t)^*\Phi(x))$ and this also holds for our $v \in \widetilde{H}^1$, which is itself an element of $V$.
We have that
\begin{equation}\label{Eq:lie_der_v}
\del_t|_{t=0} \rho(t)^*v = \left. \left(\m{L}_{Y_{\varphi_E}} \mbf{v}\right)\right|_{0},
\end{equation}
where $\mbf{v}$ is a vector field on $V$ such that $\mbf{v}|_{0} = v$.  By abuse of notation, we will often denote this derivative by $\m{L}_{Y_{\varphi_E}}v$.

\begin{lemma}
For $v \in \widetilde{H}^1$, $\nu_t(v) = \nu(v_t)$.
\end{lemma}
\begin{proof}
Again, this follows from equivariancy. Consider again the moment map $\mu(x) = S(\omega, \Phi(x))$ and denote by $\mu_t$ the map
\begin{equation*}
\begin{aligned}
\mu_t : V &\to \f{k} \\
x &\mapsto S(\omega, \rho(t)^*\Phi(x)).
\end{aligned}
\end{equation*}
Because $\Phi$ and $\mu_t$ are (locally) $K^{\bb{C}}$-invariant, and in light of the above computation, we obtain
\begin{equation*}
S(\omega, \rho(t)^*\Phi(x)) = S(\omega, \Phi(\rho(t)^*x)) = \mu(\rho(t)^*x) = \rho(t)^*\mu(x).
\end{equation*}
Now let $v \in T_0V = \widetilde{H}^1$. Then
\begin{equation*}
\mu_t(sv) = \rho(t)^*\mu(sv) = \rho(t)^*\left[ \frac{s^2}{2} \left.\frac{\dd}{\dd s}\right|_{s=0} \nu(v) + O(s^{3})\right].
\end{equation*}
But also
\begin{equation*}
\mu_t(sv) = \frac{s^2}{2} \left.\frac{\dd}{\dd s}\right|_{s=0} \nu_t(v) + O(s^{3}).
\end{equation*}
Thus $\nu_t(v) = \rho(t)^*\nu(v) = \nu(\rho(t)^*v)$, as claimed.
\end{proof}

Using this lemma and equation \eqref{Eq:lie_der_v}, the derivative \eqref{Eq:der_nu_t_v_t} becomes
\begin{equation*}
\del_t|_{t=0} \nu_t(v_t) = 2\ \dd_v\nu\left(\m{L}_{Y_{\varphi_E}}v\right).
\end{equation*}
Thus using the definition of moment map we can compute the linearisation of the map $\nu$ as follows. Letting $\psi \in \f{k}$,
\begin{equation*}
\dd_v \langle \nu, \psi \rangle (\m{L}_{Y_{\varphi_E}}v) = \Omega_0 \left(\m{L}_{Y_{\varphi_E}}v, \m{L}_{X_\psi} v\right),
\end{equation*}
where $X_{\psi} = \mrm{grad}^{\omega}\psi$. Recall the linearised infinitesimal action induced by $\psi \in \f{k}$ defined in \eqref{Eq:linearised_infinitesimal_action}, and denoted $A_\psi$. We showed in \eqref{Eq:property_of_A_and_Lie_der} that
\begin{equation*}
A_\psi v = -\left(\m{L}_{X_\psi} \mbf{v}\right)|_{0}.
\end{equation*}
Thus using the definition of $\Omega_0$,
\begin{equation}\label{Eq:linearisation_nu}
\begin{aligned}
\dd_v \langle \nu, \psi \rangle \left(\m{L}_{Y_{\varphi_E}}v\right) &= \int_M \langle J_0 \dd_0 \Phi \left(\m{L}_{Y_{\varphi_E}}v\right), \dd_0 \Phi\left(\m{L}_{X_\psi} v\right)\rangle_{\omega} \ \omega^m \\
&= \int_M \langle \dd_0 \Phi \left(\m{L}_{X_{\varphi_E}}v\right), \dd_0 \Phi\left(\m{L}_{X_\psi} v\right)\rangle_{\omega} \ \omega^m \\
&= \int_M \langle \dd_0 \Phi \left(A_{\varphi_E}v\right), \dd_0 \Phi\left(A_\psi v\right)\rangle_{\omega} \ \omega^m,
\end{aligned}
\end{equation}
where $\langle \cdot, \cdot \rangle_{\omega}$ is the inner product induced by the Riemannian metric $g(\omega, J_0)$.

\subsection{Linearisation of the optimal symplectic connection equation}\label{Subsec:linearisation_OSC}
Let us now return to the fibration setting.
Letting $\varphi, \psi \in C^\infty(E)$, by applying \eqref{Eq:linearisation_nu} and the fact that the map $\nu_\pi$ defined in \ref{Def:mu_pi_nu_pi} is defined fibrewise, we obtain
\begin{equation}\label{Eq:linearised_relative_map_nu}
\langle \dd_v \nu (\m{L}_{Y_{\varphi}}v), \psi \rangle = \int_X \langle \dd_0 \Phi \left(A_{\varphi}v\right), \dd_0 \Phi\left(A_\psi v\right) \rangle_{\omega_F} \omega_F^m \wedge \omega_B^n.
\end{equation}
Here, the map $A_\psi$ acts vertically, because it is induced by the infinitesimal action of the group of holomorphic isometries of every fibre. By using equation \eqref{Eq:linearised_relative_map_nu}, we obtain the following result.

\begin{lemma}\label{Lemma:linearisation}
Let $\widehat{\m{L}}$ be the linearisation of the equation \eqref{Eq:genOSC} at a solution, composed with the projection $p_E$. Then
\begin{equation*}
\langle\widehat{\m{L}}(\varphi), \psi \rangle = \int_X \langle \m{R}\varphi, \m{R}\psi \rangle_{\omega_F + \omega_B} \omega_F^m\wedge \omega_B^n + \lambda \int_X \langle \dd_0 \Phi \left(A_{\varphi}v\right), \dd_0 \Phi\left(A_\psi v\right) \rangle_{\omega_F} \omega_F^m \wedge \omega_B^n,
\end{equation*}
where $\m{R}$ is the operator \eqref{Eq:operator_R}, which gives the linearisation of the optimal symplectic connection equation at a solution.
\end{lemma}
From this expression it follows that $\widehat{\m{L}}$ is self adjoint.

We now study the kernel of $\widehat{\m{L}}$. Since $v$ is fixed in our setting, we can define the maps
\begin{equation}\label{Eq:relative_linearised_infinitesimal_action}
\begin{aligned}
A : C^\infty(E) &\to \widetilde{H}^1_{\m{V}} &\quad \text{and}\qquad \m{A} : C^\infty(E) &\to T_{J_0}^{0,1}\scr{J}_\pi\\
\psi &\mapsto A_\psi v &\quad \psi &\mapsto \dd_0\Phi\left(A_\psi v\right).
\end{aligned}
\end{equation}
\begin{lemma}\label{Lemma:ker_A}
A function $\psi \in C^\infty(E)$ is in the kernel of $A$ if and only if $\psi$ is a fibrewise holomorphy potential with respect to all $J_s$, i.e. $\psi \in C^\infty(E, J_s)$.
\end{lemma}
\begin{proof}
Let $\psi \in \mrm{ker}A$ and take $x_s \in V$ is such that $x_0 = 0$ and $\dot{x}_0 = v$. Then
\begin{equation*}
\begin{aligned}
0 = A_\psi(v) &= \left.\frac{\dd}{\dd t}\right|_{t=0} \dd_0 (x \mapsto \mrm{exp}(t\psi) \cdot x)(v)\\
&= \left.\frac{\dd}{\dd t}\right|_{t=0}\left.\frac{\dd}{\dd s}\right|_{s=0} \mrm{exp}(t\psi) \cdot x_s \\
&= \left.\frac{\dd}{\dd s}\right|_{s=0}\left.\frac{\dd}{\dd t}\right|_{t=0} \left(\rho^{\xi_\psi}(t)\right)^*x_s \\
&= \left.\frac{\dd}{\dd s}\right|_{s=0}\m{L}_{\xi_\psi} J_{x_s}.
\end{aligned}
\end{equation*}
where by $\rho^{\xi_\psi}(t)$ we denote the flow of the vertical vector field $\xi_{\psi} = \mrm{grad}^{\omega_F}\psi$, and all the equalities hold fibrewise since the Hamiltonian action we consider is a fibrewise action. So $\m{L}_{\xi_\psi} J_{x_s}$ is fibrewise constant, i.e.
\begin{equation*}
\left(\m{L}_{\xi_\psi} J_{x_s}\right)_{\m{V}} = \left(\m{L}_{\xi_\psi} J_{0}\right)_{\m{V}} = 0
\end{equation*}
for all $s$. This can be rephrased as
\begin{equation*}
\delbar_{s,\m{V}} \xi_\psi = 0
\end{equation*}
for all $s$, where $\delbar_{s,\m{V}}$ is the vertical $\delbar$-operator computed with respect to $J_s$. Notice that $X_\psi$ is a real vector field which corresponds (under the isomorphism between real vector fields and $(1,0)_s$ vector fields) to $J_s \nabla_{s, \m{V}}^{1,0}\psi$, where $\nabla_{s, \m{V}}\psi$ denotes the vertical vector field which on each fibre is the Riemannian gradient with respect to the fibrewise metric induced by $(\omega_F, J_s)$. Since $J_s$ is integrable,
\begin{equation*}
\left(\m{L}_{X_\psi} J_s\right)_{\m{V}} = \left(\m{L}_{J_s \nabla_{s, \m{V}}^{1,0}\psi} J_s\right)_{\m{V}} = \left(J_s \m{L}_{\nabla_{s, \m{V}}^{1,0}\psi} J_s\right)_{\m{V}} = 0,
\end{equation*}
so $\psi$ is a fibrewise holomorphic potential for $J_s$.
\end{proof}

\begin{proposition}\label{Prop:kernel_linearisation}
The kernel of $\widehat{\m{L}}$ is given by
\begin{equation*}
\mrm{ker}\widehat{\m{L}} = \set*{\psi \in C^\infty(E, J_0) | \delbar_s (\nabla_{s, \m{V}}^{1,0}\psi )= 0 \ \forall s},
\end{equation*}
i.e. the functions in the kernel are those fibrewise $J_0$-holomorphy potentials which are global holomorphy potentials with respect to all $J_s$.
\end{proposition}
\begin{proof}
Since $\Phi$ is an embedding, $\dd_0\Phi$ is injective, so for $\psi \in C^\infty(E)$, $\psi \in \mrm{ker}\widehat{\m{L}}$ if and only if $\psi \in \mrm{ker}{\m{R}}$ and $\psi \in \mrm{ker}{A}$.

As seen in \S \ref{Subsec:OSC_classic}, the kernel of $\m{R}$ consists of fibrewise holomorphy potentials which are also global holomorphy potentials.
Thus $\psi \in C^\infty(E)$ lies in $\mrm{ker}\widehat{\m{L}}$ if and only if
\begin{equation*}
\delbar_B \nabla^{1,0}_{\m{V}} \psi = 0 \quad \text{and} \quad \delbar_{s, \m{V}} \nabla^{1,0}_{s, \m{V}} \psi = 0.
\end{equation*}
as shown in Lemma \ref{Lemma:ker_A}. From these two conditions, and in light of Lemma \ref{Lemma:deformations_preserve_HV}, which implies that $\delbar_B$ does not depend on $s$, we have
\begin{equation*}
\delbar_s \mrm{grad}^{\omega_F} \psi = \delbar_{s,\m{V}} \mrm{grad}^{\omega_F} \psi + \delbar_B \mrm{grad}^{\omega_F} \psi = 0,
\end{equation*}
as claimed.
\end{proof}
\begin{remark}\label{Rmk:Aut_intersection}
In \cite[\S 4.1]{DervanSektnan_OSC3} it is explained that the kernel of the operator $\m{R}$ is given by the Lie algebra of the group $\mrm{Aut}(\pi)$ of automorphisms of the projection, described in Definition \ref{Def:aut(pi)}. In our case, the kernel of the linearisation $\widehat{\m{L}}$ is the intersection
\begin{equation*}
\mrm{ker}\ \widehat{\m{L}} = \mrm{Lie}(\mrm{Aut}(\pi_s)) \cap \mrm{Lie}(\mrm{Aut}(\pi)),
\end{equation*}
where we denote by $\pi : X \to B$ the central fibration and we view $\{J_s\}$ as a family of complex structures on the same underlying smooth manifolds, compatible with the projection and with $\omega_X$ (see the end of Section \ref{Subsec:family_submersions}).
\end{remark}

We wish to see that $\widehat{\m{L}}$ is elliptic as a differential operator on the global sections of $E\to B$. Let us split $\m{A}$ in \eqref{Eq:relative_linearised_infinitesimal_action} as the composition of the two operators
\begin{equation*}
\begin{aligned}
A_1 : C^\infty(E) &\to \Upgamma(\m{V}) \\
\varphi &\mapsto \mrm{grad}^{\omega_F}\varphi
\end{aligned}
\end{equation*}
and
\begin{equation*}
\begin{aligned}
A_2 : \Upgamma(\m{V}) &\to T_{J_0}\scr{J}_\pi \\
Y &\mapsto -(\m{L}_{Y}v).
\end{aligned}
\end{equation*}
To give a local expression, we make use of \emph{Riemannian} coordinates, and we denote again the vertical coordinates with the letters $a, b, c, \dots$ and the horizontal coordinates with the letters $i, j, k, \dots$, as in Section \ref{Sec:deformations_of_fibrations}. We have:
\begin{equation*}
\begin{aligned}
&(A_2(Y))^a_{\ b} = -\left(\m{L}_{Y}v\right)^a_{\ b} = -Y^c\del_cv^a_{\ b} - v^a_{\ c} \del_b Y^c + v^c_{\ b}\del_c Y^a \\
&(A_2(Y))^a_{\ j} = (A_2(Y))^{a}_{\ c} (\omega_F)^{dc}(\omega_X)_{dj}
\end{aligned}
\end{equation*}
where the second expression follows from Lemma \ref{Lemma:deformation_coord_vertical_horizontal}.
Thus when $\varphi \in C^\infty(E)$ and $Y = \omega_F^{de}\del_e \varphi \del_d$,
\begin{equation*}
\begin{aligned}
&(\m{A}(\varphi))^a_{\ b} =- v^a_{\ c} \del_b (\omega_F^{cd}\del_d\varphi) + v^c_{\ b}\del_c (\omega_F^{ad}\del_d\varphi) + T_1(\varphi) \\
&(\m{A}(\varphi))^a_{\ j} = - v^a_{\ c} \del_b (\omega_F^{cd}\del_d\varphi)(\omega_F)^{eb}(\omega_X)_{ej} + v^c_{\ b}\del_c (\omega_F^{ad}\del_d\varphi) (\omega_F)^{eb}(\omega_X)_{ej}+ T_1'(\varphi),
\end{aligned}
\end{equation*}
where $T_1(\varphi)$ and $T_1'(\varphi)$ are terms involving first order vertical derivatives of $\varphi$. Thus we see that $\m{A}$ is a second order differential operator, and all the derivatives of $\varphi$ involved are vertical. The adjoint of $A_1$ is given by
\begin{equation*}
A_1^* (Y) = \mrm{div}(J_0Y).
\end{equation*}
Indeed, we can compute the divergence with respect to any K\"ahler metric $g$ whose K\"ahler form restricts vertically to $\omega_F$, and the result depends only on the vertical part:
\begin{equation*}
\begin{aligned}
\langle A_1\varphi, Y \rangle_{L^2} &= \int_X{g_{ac}}\ \omega_F^{ab} \ \del_b\varphi \ Y^c \ \dd Vol_g = \int_X -ig_{ac}g^{ab} \ (\nabla_b \varphi) \ Y^c \ \dd Vol_g \\
&= \int_X -i (\nabla_c\varphi) \ Y^c \ \dd Vol_g = \int_X \varphi \ \nabla_c(iY^c) \ \dd Vol_g = \langle \varphi, A_1^*Y \rangle_{L^2}.
\end{aligned}
\end{equation*}

To compute the adjoint of $A_2$ we make use of the following lemma.
\begin{lemma}
Let $Q \in \widetilde{H}^1_{\m{V}}$ and let $g_F$ be the vertical Riemannian metric induced by $(\omega_F, J_0)$. Then
\begin{equation*}
g_F(\m{L}_Yv, Q) = g_F(Q, \nabla v(Y)) + g_F(vQ-Qv, \nabla Y).
\end{equation*}
\end{lemma}
The proof of the lemma is obtained by computing the different quantities in Riemannian coordinates \cite[\S 4.2]{ScarpaStoppa_HcscK}.

In light of the lemma, the adjoint to $A_2$ can be formally written as
\begin{equation*}
A_2^*(Q) = -(\nabla v)^*Q - \nabla^*([v,Q]).
\end{equation*}
If $Q = \m{A}(\varphi)$, the first term is of order 3.
So we have:
\begin{equation*}
\m{A}^*\m{A}(\varphi) = -\mrm{div}\left( J_0\nabla_{\m{V}}^*(v(\m{L}_{X_{\varphi}}v)_{\m{V}}-(\m{L}_{X_\varphi} v)_{\m{V}}v)\right)+\text{lower order terms}.
\end{equation*}
From this expression, we see that all the quantities involved are vertical. This means that, as an operator on the global sections of the vector bundle $E$, the operator
\begin{equation*}
\m{A}^*\m{A} : C^\infty(E) \to C^\infty(E)
\end{equation*}
is of order $0$. Indeed, let us denote by $r$ the rank of $E$ and consider a local frame $h_1, \dots, h_r$ of $E$. Then we can write a local section $h = \sum_i f_ih_i$, with $f_i \in C^\infty(B)$. Then
\begin{equation*}
\m{A}^*\m{A} (h) = \sum_i f_i \m{A}^*\m{A}(h_i).
\end{equation*}
Thus, as an operator on the global sections $C^\infty(E)$, the operator $\widehat{\m{L}}$ is elliptic, since $\m{R}^*\m{R}$ is from \cite[\S 4]{DervanSektnan_OSC1} and $\m{A}^*\m{A}$ is of lower order. We have established the following:
\begin{theorem}\label{Thm:linearisation_genOSC}
Let $\widehat{\m{L}}$ be the linearisation of the optimal symplectic connection equation \eqref{Eq:genOSC}. Then $\widehat{\m{L}}$ is an elliptic operator of order two on the global sections of $E$ which is self-adjoint and whose kernel consists of fibrewise $J_0$-holomorphy potentials which are also global $J_s$-holomorphy potentials for all $s$.
\end{theorem}

\subsection{Approximate solutions in the case of discrete automorphism group}\label{Subsec:approx_solution_discrete_autom}
Let $(\m{X}, \m{H})\to B\times S$ a family of submersions with central fibre the fibration $(X,H)\to (B,L)$ as before. In this section we construct approximate constant scalar curvature K\"ahler metrics on the total space of $\pi_s:(X_s, H_s){\to} (B, L)$, where $(X_s, H_s)$ is a deformation of a fibration $\pi: (X, H)\to (B,L)$ whose fibres are cscK and where we assume that $(X_s, H_s)$ admits an optimal symplectic connection.
We do so by using an \emph{adiabatic limit}, such as in \cite{Fine_cscK_fibrations,DervanSektnan_OSC1}.

We make the following assumptions:
\begin{enumerate}[label = (\roman*)]
\item The base form $\omega_B \in c_1(L)$ is twisted cscK with respect to the pull-back via the moduli map $q$ of the Weil-Petersson metric, as in Definition \ref{Def:twisted_cscK_extremal};
\item \label{Item:Autq_discrete} The group $\mrm{Aut}(q)$ defined in \ref{Def:aut(q)} is discrete. As recalled in the discussion following Definition \ref{Def:aut(q)}, this implies that the linearisation at a solution of the twisted cscK equation on the base is invertible;
\item $\mrm{Aut}(X_s, H_s)$ is discrete. Thanks to Proposition \ref{Prop:kernel_linearisation}, this guarantees that the operator $\widehat{\m{L}}$ is invertible and also that the global Lichnerowicz operator on $X_s$ with respect to $\omega_k$ is invertible.
\end{enumerate}

Let $k \gg 0$ be such that
\begin{equation*}
\omega_k = \omega_X + k\omega_B
\end{equation*}
is a K\"ahler metric on $X$, and let $s^2 = \lambda k^{-1}$ for $\lambda >0$.
\begin{theorem}\label{Thm:existence_cscK_metrics}
With the assumptions listed above, let $\omega_X$ be an optimal symplectic connection for the family $\m{X}\to B \times S$. Then for all $k\gg0$ there exists a constant scalar curvature K\"ahler metric on $X_s$ for $s\ne 0$, in the class $[\omega_X ]+ k[\omega_B]$.
\end{theorem}
In the proof of Theorem \ref{Thm:existence_cscK_metrics}, we will relate $s$ and $k$ as above, namely $s^2 = \lambda k^{-1}$, so we will sometimes denote also the corresponding complex structure by $J_k$. Since all the $J_s$ are isomorphic, Theorem \ref{Thm:existence_cscK_metrics} still gives the existence of a cscK metric in each adiabatic class for all $J_s$.
The adiabatic limit technique consists in constructing inductively approximated solutions, which have constant scalar curvature up to a certain order in $k^{-1/2}$, then using the implicit function theorem to perturb an approximate solution to a genuine solution. The following result establishes the approximate solution.
\begin{proposition}\label{Prop:approximate_solutions}
With the assumptions listed above, for all $k\gg0$ and for each $r$ there exist functions
\begin{equation*}
f_{B,1}, \dots, f_{B,r} \in C^\infty(B) \qquad f_{E,1}, \dots, f_{E,r} \in C^\infty(E) \qquad f_{R,1}, \dots, f_{R,r} \in C^\infty(R)
\end{equation*}
and constants
\begin{equation*}
\widehat{S}_1, \dots, \widehat{S}_r
\end{equation*}
such that the K\"ahler potentials
\begin{equation*}
h^B_{k,r} = \sum_{j=2}^r {k^{j-2}}{f_{B,j}} \qquad h^E_{k,r} = \sum_{j=2}^r {k^{(j-1)/2}}{f_{E,j}} \qquad h^R_{k,r} = \sum_{j=2}^r {k^{j/2}}{f_{R,j}}
\end{equation*}
satisfy
\begin{equation*}
\mrm{Scal}\left(\omega_k + i\del\delbar\left(h^B_{k,r} + h^E_{k,r} + h^R_{k,r}\right), J_k\right) = \widehat{S}_b + \sum_{j=1}^r {k^{j/2}}{\widehat{S}_j} + O\left(k^{(-r-1)/2}\right).
\end{equation*}
\end{proposition}
\begin{proof}
With the hypotheses of $\omega_X$ being an optimal symplectic connection and $\omega_B$ being a twisted cscK metric on the base, we have
\begin{equation}\label{Eq:adiabatic_limit_step_0}
\mrm{Scal}(\omega_k) = \widehat{S}_b + k^{-1}\left( c_B + \psi_{R,1}\right) + O\left(k^{-3/2}\right),
\end{equation}
where $\psi_{R,1} \in C^\infty(R)$.
In order to make the $k^{-1}$-term constant we add a potential $k^{-1}f \in C^\infty(R)$ to $\omega_k$. Then
\begin{equation*}
\mrm{Scal}(\omega_k + k^{-1}i\del\delbar f) = \widehat{S}_b + k^{-1}\left( c_B + \psi_{R,1} - \m{D}_{\m{V}}^*\m{D}_{\m{V}}f\right) + O\left(k^{-3/2}\right),
\end{equation*}
where the linearisation of the scalar curvature to order 0 in $k$ coincides with (minus) the Lichnerowicz operator with respect to the complex structure $J_0$, since the scalar curvature is constant in order 0, and the higher order terms fall into $O\left(k^{-3/2}\right)$. Since $\m{D}_{\m{V}}^*\m{D}_{\m{V}}$ is a fibrewise elliptic differential operator and $C^\infty(R)$ is orthogonal to its kernel, we can find a solution $f_{R,1}$ of
\begin{equation}\label{Eq:adiabatic_limit_step_1}
\psi_{R,1} - \m{D}_{\m{V}}^*\m{D}_{\m{V}}f = constant.
\end{equation}
Summing up, we have proved step $n=1$ of Proposition \ref{Prop:approximate_solutions}, with $f_{B,1} = 0 = f_{E,1}$. We define
\begin{equation*}
\omega_{k,1} = \omega_k + k^{-1}i\del\delbar f_{R,1} 
\end{equation*}
such that
\begin{equation*}
\mrm{Scal}(\omega_{k,1}) = \widehat{S}_b + k^{-1} \widehat{S}_1 + O\left(k^{-3/2}\right).
\end{equation*}
To proceed with the approximate solutions, we need the linearisation of the scalar curvature at a metric $(\omega_{k,r}, J_k)$.
\begin{lemma}\label{Lemma:linearisation_scal_kn}
The linearisation of the scalar curvature of $\omega_{k,r}$ satisfies
\begin{equation*}
\m{L}_{k,r} = -\m{D}_{\m{V}}^*\m{D}_{\m{V}} + k^{-1}D_1 + {k^{-3/2}}D_{3/2} + k^{-2}D_2 + O\left(k^{-5/2}\right),
\end{equation*}
where
\begin{enumerate}
\item ${D}_{\m{V}}^*\m{D}_{\m{V}}$ is the vertical Lichnerowicz operator with respect to the complex structure $J_0$;
\item If $f \in C^\infty(B)$, $D_{j/2}(f)=0$ for all $j$;
\item If $f \in C^\infty(B)$, $D_1(f) = 0$ and
\begin{equation*}
\int_{X/B} D_2(f) \omega_X^m\wedge\omega_B^n = -\m{L}_\alpha (f),
\end{equation*}
where $\m{L}_\alpha$ is the linearisation of the twisted cscK equation on the base, with twisting the Weil-Petersson form $\alpha_{WP}$, at a solution, defined in \eqref{Eq:linearisation_twisted_eq}.
\item  If $f \in C^\infty(E)$, then
\begin{equation*}
p_E \circ D_1(f) = -p_E \circ \widehat{\m{L}}(f).
\end{equation*}
\end{enumerate}
\end{lemma}
\begin{proof}[Proof of the Lemma]
Let us distinguish the parameter $s$ of the deformation of the complex structure from the parameter $k$ of the polarisation. Consider the case $n=0$, so that we compute the scalar curvature of the metric $(\omega_k, J_s)$.  Then
\begin{equation}\label{Eq:linearisazion_scalar_curvature_ks}
\m{L}_k = \m{L}_{k,0} + O(s),
\end{equation}
where $\m{L}_{k,0}$ is the linearisation of the scalar curvature of $(\omega_k, J_0)$. 
In \cite[Proposition 4.11]{DervanSektnan_OSC1} it is proven that
\begin{equation*}
\m{L}_{k,0} = -{D}_{\m{V}}^*\m{D}_{\m{V}} + k^{-1}D_1' +k^{-2}D_2'+ O\left( k^{-3} \right),
\end{equation*}
from which we see that the term of order zero is indeed the vertical $J_0$-Lichnerowicz operator. This proves claim $(1)$. By imposing the relation $s^2 = \lambda k^{-1/2}$ we see that the $O(s)$-term in \eqref{Eq:linearisazion_scalar_curvature_ks} admits an expansion in powers of $k^{-1/2}$:
\begin{equation*}
k^{-1}D_1'' + {k^{-3/2}}D_{3/2}'' + k^{-2}D_2'' + O\left(k^{-5/2}\right).
\end{equation*}

Claim $(2)$ follows from the fact that the deformation of the complex structure is vertical, thus all the terms involved in the expansion of the scalar curvature coming from the deformation do not have a $C^\infty(B)$-component.

Claims $(3)$ and $(4)$ follow as in \cite[Proposition 4.11]{DervanSektnan_OSC1}.
\end{proof}
The proof of Proposition \ref{Prop:approximate_solutions} now goes by induction, using Lemma \ref{Lemma:linearisation_scal_kn}.
We explain in detail steps $n=\frac{3}{2}$ and $n=2$. We start from the expansion
\begin{equation*}
\mrm{Scal}(\omega_{k,1}) = \widehat{S}_b + k^{-1} \widehat{S}_1 + k^{-3/2}(\psi_{E, 3/2} + \psi_{R, 3/2}) + O\left(k^{-2}\right).
\end{equation*}
We add a potential $k^{-1/2}f_{E}$ to $\omega_{k,1}$. Thus we have
\begin{equation*}
\mrm{Scal}\left(\omega_{k,1} + k^{-1/2}i\del\delbar f_E\right) = \widehat{S}_b + k^{-1}\widehat{S}_1 + k^{-2/3}\left(\psi_{E,3/2} + D_1(f) + \psi_{R,3/2}\right) + O\left(k^{-2}\right).
\end{equation*}
Using Lemma \ref{Lemma:linearisation_scal_kn}, our hypothesis on the automorphism group of $(X_s, H_s)$ and the fact that the linearisation $\widehat{\m{L}}$ of the optimal symplectic connection equation at a solution is elliptic, as proved in Theorem \ref{Thm:linearisation_genOSC}, we can find $f_{E,2}$ such that
\begin{equation*}
\psi_{E,2/3} + p_E\circ D_1(f_E) = constant.
\end{equation*}
This makes the $C^\infty(E)$-term constant to order $k^{-3/2}$. We next add a potential $k^{-3/2}f_R \in C^\infty(R)$ and we obtain
\begin{equation*}
\begin{split}
\mrm{Scal}\left(\omega_{k,1} + i\del\delbar\left(k^{-1/2}f_{E,3/2} + k^{-3/2}f_R\right)\right) &= \widehat{S}_b + k^{-1}\widehat{S}_1 + \\
&+ k^{-3/2}\left(c_{E,3/2} + \psi'_{R,3/2} - \m{D}_{\m{V}}^*\m{D}_{\m{V}}f_R\right) + O\left(k^{-2}\right).
\end{split}
\end{equation*}
Once again, using the fibrewise ellipticity of $\m{D}_{\m{V}}^*\m{D}_{\m{V}}$ and the fact that $C^\infty(R)$ is orthogonal to its kernel, we obtain a solution $f_{R,3/2}$ of the equation
\begin{equation*}
\psi'_{R,2} - \m{D}_{\m{V}}^*\m{D}_{\m{V}}f_R = constant.
\end{equation*}
Thus we have constructed a K\"ahler metric on $X_s$ constant up to order $k^{-3/2}$:
\begin{equation*}
\omega_{k,3/2} = \omega_{k,1} + i\del\delbar\left(k^{-1/2}f_{E,3/2} + k^{-3/2}f_{R,3/2}\right).
\end{equation*}

As for the step $n=2$, we explain how to deal with the $C^\infty(B)$-term. We add a potential $f_{B}$ to $\omega_{k,3/2}$, which amounts to adding a potential $k^{-1}f_B$ to $\omega_B$. Since the scalar curvature of the base affects the order $k^{-1}$-term and not the order zero term, the combined effect on the linearisation is of order $k^{-2}$. This allows us to write
\begin{equation*}
\begin{split}
\mrm{Scal}(\omega_{k,3/2} + i\del\delbar f_B) = \widehat{S}_b + k^{-1}\widehat{S}_1 &+ k^{-3/2} \widehat{S}_{3/2} +\\
&+ k^{-2}\left(\psi_{B,2} - D_2(f_B) + \psi_{E,2} + \psi_{R,2}\right) + O\left(k^{-5/2}\right).
\end{split}
\end{equation*}
Thanks to Lemma \ref{Lemma:linearisation_scal_kn} and to our hypothesis on the automorphism group of the moduli map,
\begin{equation*}
\psi_{B,2} - p_B\circ D_2(f_B) = constant
\end{equation*}
admits a solution, which we denote $f_{B,2}$. This makes the $C^\infty(B)$-term constant to order $k^{-2}$.

The corrections to the $C^\infty(E)$-term and to the $C^\infty(B)$-term now work exactly as in the case $n=3/2$.
\end{proof}

Notice that the order is important: one can make the $C^\infty(E)$-term constant without affecting the $C^\infty(B)$-term, but it cannot work the other way around, and similarly for the $C^\infty(R)$-term.

\begin{remark}
The very first step of the approximate solution procedure, which the expansion \eqref{Eq:adiabatic_limit_step_0}, comes from the fact that in Proposition \ref{Prop:Kuranishi_map_cscK} we have modified the Kuranishi map $\Phi$ in order to meet the requirement that $\mrm{Scal}_{\m{V}}(\omega_X, \Phi(x))$ is a section of $E$, for $x \in V$. If we do not deform the Kuranishi map in this way, we can write the vertical scalar curvature as the sum of the projection onto $C^\infty(E)$ and the projection onto $C^\infty(R)$. The $C^\infty(E)$-part is the map $\mu_\pi$ defined in \ref{Def:mu_pi_nu_pi}, while the $C^\infty(R)$-part introduces a term of order $k^{-1/2}$ in the expansion \eqref{Eq:adiabatic_limit_step_0}, which then becomes
\begin{equation*}
\mrm{Scal}(\omega_k) = \widehat{S}_b + k^{-1/2} \ \psi_{R,0}+ k^{-1}\left( c_B + \psi_{R,1}\right) + O\left(k^{-2}\right).
\end{equation*}
We can get rid of this term by adding a potential $k^{-1/2}i\del\delbar\varphi_{R,0}$ to $\omega_k$, as in equation \eqref{Eq:adiabatic_limit_step_1}.
Indeed, the linearisation given by Lemma \ref{Lemma:linearisation_scal_kn} of the scalar curvature acquires an extra term $\sqrt{k}D_{1/2}$, which is non-zero only on $C^\infty(R)$, so it does not affect the $C^\infty(E)$ and $C^\infty(B)$ parts in the $k^{-1}$-term.
\end{remark}

\subsection{Approximate solutions in the presence of automorphisms}\label{Subsec:approx_sol_extremal}
In this section, we allow the base and the total space to have automorphisms. As before let $\widehat{\pi}: (\m{X},\m{H}) \to (B,L) \times S$ be a degeneration of the fibration $\pi_Y: (Y, H_Y)\to B$ to $\pi:(X, H)\to B$. Let $\omega_X \in c_1(H)$ be a relatively cscK metric on $X$; since $Y$ is a small deformation of $X$, $c_1(H) = c_1(H_Y)$, so we can assume that $\omega_X$ is relatively K\"ahler on $Y$ (as explained in the end of Section \ref{Subsec:family_submersions}).

Recall from Definition \ref{Def:extremal_OSC} that $\omega_X$ is an extremal symplectic connection on $Y$ if
\begin{equation*}
\widehat{\m{L}}\left(p_E (\Theta(\omega_X, J_0)) + \frac{\lambda}{2}\nu \right) = 0,
\end{equation*}
so that the function
\begin{equation*}
h_1 := p_E \left(\Theta(\omega_X, J_0)\right) + \frac{\lambda}{2}\nu
\end{equation*}
is a holomorphy potential for the complex structure of $Y$.

We make the following hypotheses concerning the groups of automorphisms $\mrm{Aut}(\pi_Y)$ and $\mrm{Aut}(q)$ defined in \ref{Def:aut(pi)} and \ref{Def:aut(q)}:
\begin{enumerate}[label=(\roman*)]
\item There is an action of $\mrm{Aut}(\pi_Y)$ on $(\m{X}, \m{H})$ which is equivariant with respect to the projection onto $S$, meaning that it acts on each $X_s$ as a subgroup of automorphisms of $(X_s, H_s)$. Since the action extends to the central fibration, this assumption allows us to view $\mrm{Aut}(\pi_Y)$ as a subgroup of $\mrm{Aut}(\pi)$. In particular, recall from Remark \ref{Rmk:Aut_intersection} that $\mrm{ker} \widehat{\m{L}} = \mrm{Lie}(\mrm{Aut}(\pi_Y))\cap \mrm{Lie}(\mrm{Aut}(\pi))$. With this assumption, we obtain
\begin{equation*}
\mrm{Ker}\ \widehat{\m{L}} = \mrm{Lie}(\mrm{Aut}(\pi_Y)),
\end{equation*}
and $h_1$ is a holomorphy potential also on $X$.
\item All automorphisms of the moduli map $q$ lift to $(Y, H_Y)$.
\end{enumerate}
The first hypothesis is motivated by the analogous definition of test configurations which are equivariant with respect to the automorphisms of the fibres, which are used to test K-polystability of polarised manifolds.

As a K\"ahler metric on the base, we require that $B$ admits a twisted extremal metric, with twisting form the Weil-Petersson form $\alpha_{WP}$ \eqref{Eq:WPform}, i.e.
\begin{equation*}
\mrm{Scal}(\omega_B) - \Lambda_{\omega_X}\alpha_{WP} = b_1 \in \mrm{ker}\m{D}_B,
\end{equation*}
where $\m{D}_B$ is the Lichnerowicz operator on the base.

\begin{theorem}\label{Thm:existence_extremal_metrics}
With the assumptions listed above, let $\omega_X$ be an extremal symplectic connection for the family $\m{X}\to B \times S$. Then for all $k\gg0$ there exists an extremal K\"ahler metric on $X_s$ for $s\ne 0$, in the class $[\omega_X] + k[\omega_B]$.
\end{theorem}

\begin{remark}\label{Rmk:aut(q)lift}
Let $\widehat{g}$ be a lift of an automorphism of $q$ to $(Y, H_Y)$. We claim that $\widehat{g}$ lies in $\mrm{Aut}(X, H)$. Indeed, denoting by $J$ the complex structure of $Y$ and $J_0$ the complex structure of $X$, we have
\begin{equation*}
\dd \widehat{g} \circ J = J \circ \widehat{g}.
\end{equation*}
But since $\widehat{g}$ is an automorphism in the base direction, it is equivalent to say that
\begin{equation*}
\dd \widehat{g} \circ J_{\m{H}} = J_{\m{H}} \circ \widehat{g},
\end{equation*}
where $J_{\m{H}}$ is the horizontal part of $J$. Now, $J_{\m{H}} = (J_0)_{\m{H}}$, since the deformation of the complex structure which we are considering is only in the vertical direction. Thus $\widehat{g}$ is a lift of an automorphism of $B$ to $X$ as well.
\end{remark}

\begin{definition}
We denote the group of automorphisms of $(Y, H_Y)$ which are also automorphisms of $(X, H)$ as $\mrm{Aut}(Y/X, H_Y)$.
\end{definition}
In light of this definition we have the inclusion $\Aut(\pi_Y) \subseteq \Aut(Y/X, H_Y)$ and, if $\widehat{\Aut}(q)$ is a lift of $\Aut(q)$ to (Y), then $\widehat{\Aut}(q) \subseteq \Aut(Y/X, H_Y)$.
Thus we can recover the following result from \cite[Proposition 3.14]{DervanSektnan_OSC1}.

\begin{lemma}\label{Lemma:ses}
Suppose that all automorphisms of $q$ lift to $Y$. Then there is a short exact sequence
\begin{equation*}
0 \to \mrm{Lie}(\mrm{Aut}(\pi_Y)) \to \mrm{Lie}(\mrm{Aut}(Y, H_Y)) \to \mrm{Lie}(\mrm{Aut}(q)) \to 0.
\end{equation*}
\end{lemma}

\begin{remark}\label{Rmk:torus}
Let us denote by $\xi_E$ the holomorphic vector field on $Y$ which arises from the extremal symplectic connection condition:
\begin{equation*}
\xi_E = J_s\nabla_{\m{V}}\left(p_E (\Theta(\omega_X, J_0)) + \frac{\lambda}{2}\nu\right),
\end{equation*}
and $\xi_q$ the holomorphic vector field on $B$ which arises from the twisted extremal condition:
\begin{equation*}
\xi_q = J_B\nabla_B (\mrm{Scal}(\omega_B) - \Lambda_{\omega_X}\alpha_{WP}).
\end{equation*}
By our assumptions, $\xi_E$ is a holomorphy potential on $X$, and $\xi_q$ lifts to a holomorphic vector field on $Y$ (and on $X$). Nonetheless, the holomorphy potential of $\xi_q$ on $Y$ is a function $\widetilde{b}_1$ such that
\begin{equation*}
\widetilde{b}_1 = k\pi^*b_1 + O(1).
\end{equation*}
Again from Remark \ref{Rmk:aut(q)lift}, $\widetilde{b}_1$ is holomorphic potential for a lift of $\xi_q$ also on $X$.
As in \cite{DervanSektnan_OSC1}, we need to assume the following invariance properties: $\omega_X$ is invariant under the flow of $\xi_E$ and of the pull-back of $\xi_q$. In order to make this assumptions reasonable to work with, we consider a maximal torus $T_E$ in $\mrm{Aut}(\pi_Y)$ which contains the flow of $\xi_E$, and a maximal torus $T_q$ in $\mrm{Aut}(B, L)$ which contains the flow of $\xi_q$. The pull back $\widehat{T_q}$ lies in $\Aut(Y/X, H_Y)$. Then we fix a maximal torus $T$ in $\mrm{Aut}(Y, H_Y)$ which contains $T_E$ and $T_q$, and we require that $\omega_X$ is invariant with respect to $T$.
From Lemma \ref{Lemma:ses}, we obtain a splitting $\mrm{Lie}(T) = \mrm{Lie}(T_E)+ \mrm{Lie}(T_q)$, so indeed we have $T \subset \Aut(Y/X, H_Y)$ as well.

Moreover, an analogous splitting holds also for the complexification $T^{\bb{C}}$, so we can write every vector field $\xi \in \mrm{Lie}(T^{\bb{C}})$ as $\xi_E + \xi_q$. If $h_E$ is the holomorphy potential of $\xi_E$ with respect to $\omega_X$ and $h_B$ is the holomorphy potential of $\xi_q$ on the base with respect to $\omega_B$, then $h_E + k\pi_Y^*h_B$ is a holomorphy potential of $\xi$ on $Y$ (and on $X$).
\end{remark}

Define the \emph{extremal symplectic connection operator}
\begin{equation*}
\m{P}: C^\infty(Y, \bb{R})\times C^\infty(E) \to C^\infty(Y, \bb{R})
\end{equation*}
by
\begin{equation*}
\m{P}(\varphi, h_1) = p_E\left(\Theta(\omega_X + i\del\delbar\varphi, J_s)\right) + \frac{\lambda}{2}\nu_\varphi -h_1-\frac{1}{2} \langle \nabla h_1,\nabla\varphi\rangle_{\omega_F}.
\end{equation*}
The linearisation at $(h_1, 0)$ applied to $(h_1, \psi)$ is obtained, as for the extremal operator described in \eqref{Eq:extremal_operator}, as follows:
\begin{equation*}
\widehat{\m{L}}(\psi)-h_1 -\frac{1}{2}\langle \nabla h_1, \nabla \psi \rangle_{\omega_F},
\end{equation*}
where $\widehat{\m{L}}$ is the real operator of the linearisation of the optimal symplectic connection equation described in Lemma \ref{Lemma:linearisation} and the map sending $\varphi$ to $\langle \nabla h_1, \nabla \varphi \rangle_{\omega_F}$ is linear.
We can write
\begin{equation*}
\langle \nabla h_1, \nabla \psi \rangle_{\omega_F} = \frac{1}{2} \nabla h_1 (\psi) + \frac{1}{2} i J \nabla h_1(\psi),
\end{equation*}
so if we assume that $\psi$ is invariant under the torus $T$, the second term vanishes and linearisation is a real operator.

With all of these assumptions in place, we can obtain approximate solutions to the extremal equation much as in Section \ref{Subsec:approx_solution_discrete_autom}.
\begin{proposition}\label{Prop:approx_solutions_aut}
Let $(\m{X}, \m{H})\to B\times S$ be a degeneration of a smooth fibration $\pi_Y:(Y, H_Y) \to B$ to a smooth relatively cscK fibration $\pi:(X, H) \to B$, equivariant with respect to $\mrm{Aut}(\pi_Y)$. Let $\omega_X$ be an extremal symplectic connection on $X_s$, invariant under the torus $T$ described in Remark \ref{Rmk:torus}. Let $\omega_B$ a twisted extremal metric on the base, and assume that all automorphisms of $q$ lift to $Y$. Then for each $r >1$ there exist functions
\begin{equation*}
f_{B,1}, \dots, f_{B,r} \in C^\infty(B)^T, \qquad f_{E,1}, \dots, f_{E,r} \in C^\infty(E)^T, \qquad f_{R,1}, \dots, f_{R,r} \in C^\infty(R)^T,
\end{equation*}
base holomorphy potentials
\begin{equation*}
b_1, \dots, b_r \in C^\infty(B)^T,
\end{equation*}
fibre holomorphy potentials
\begin{equation*}
h_1, \dots, h_r \in C^\infty(E)^T
\end{equation*}
and a constant $c$ such that, letting
\begin{equation*}
h^B_{k,r} = \sum_{j=2}^r \frac{f_{B,j}}{k^{j-2}}, \qquad h^E_{k,r} = \sum_{j=2}^r \frac{f_{E,j}}{k^{(j-1)/2}}, \qquad h^R_{k,r} = \sum_{j=2}^r \frac{f_{R,j}}{k^{j/2}}
\end{equation*}
and
\begin{equation*}
\eta_{k, r} = c + \sum_{j=1}^r \left(\widetilde{b}_{j}k^{(-j-1)/2}+ h_j k^{-j/2}\right),
\end{equation*}
the K\"ahler metric
\begin{equation*}
\omega_{k,r} = \omega_k + i\del\delbar\left(h^B_{k,r} + h^E_{k,r} + h^R_{k,r} \right)
\end{equation*}
satisfies
\begin{equation*}
\mrm{Scal}\left(\omega_{k, r}, J_k\right) = \eta_{k,r} + \frac{1}{2} \langle\nabla\eta_{k,r}, \nabla\left(h^B_{k,r} + h^E_{k,r} + h^R_{k,r}\right) \rangle_{\omega_k}+ O\left(k^{-(r+1)/2}\right).
\end{equation*}
\end{proposition}

\subsection{Solution to the non-linear equation}

In order to have genuine solutions, we perturb $\omega_{k,r}$ to a genuine extremal metric by using a quantitative version of the implicit function theorem, as in \cite{Fine_cscK_fibrations, Bronnle_extremal_projvb, DervanSektnan_ExtremalFibrations, DervanSektnan_OSC1}. In particular, all the cited works rely on Fine's paper \cite{Fine_cscK_fibrations}, though the difference with Fine's setting is that we are considering the base and the total space to have automorphisms, so the linearised operators will have a non-trivial kernel to deal with.

\begin{theorem}\label{Thm:quantitative_implicit_function_theorem}\cite[Theorem 25]{Bronnle_extremal_projvb}
Let $F:B_1 \to B_2$ be a differentiable map of Banach spaces such that $D_0F$ is surjective with right-inverse $P$. Let
\begin{enumerate}[label=(\roman*)] 
\item $\delta'>0$ be such that the non-linear operator $(F-D_0F)$ is Lipschitz in $B_{\delta'}(0)$ with constant $\frac{1}{2\norm{P}}$, i.e. for $x_1, x_2 \in B_{\delta'}(0) \subseteq B_1$, we have
\begin{equation*}
\norm*{\left(F-D_0F\right)(x_1)-(F-D_0F)(x_2)}_{B_2} \le \frac{1}{2\norm{P}} \norm*{x_1-x_2}_{B_1};
\end{equation*}
\item $\delta = \frac{\delta'}{2\norm{P}}$.
\end{enumerate}
Then for all $ y \in B_2$ such that $\norm{y-F(0)}<\delta$, there exists $x \in B_1$ such that $F(x)=y$.
\end{theorem}

To apply the theorem to the extremal operator, one should bound both the right inverse of the linearisation and the non-linear operator. Denote by $L^2_{0,p}$ the Sobolev spaces of functions on $Y$ computed with respect to $\omega_{k,r}$, and remark that these do not depend on $k$, since the Sobolev norms are equivalent for different values of $k$ \cite[Remark 30]{Bronnle_extremal_projvb}.

Let $\f{t}$ be the Lie algebra of $T$, where $T$ is the torus of automorphisms described in Remark \ref{Rmk:torus}. Let $\overline{\f{t}}$ be the set of holomorphy potentials whose flow lies in $T$. We denote by $(L^2_{0,p})^T$ the space of $T$-invariant functions in $L^2_{0,p}$.

For each $k,r$ denote by $\gamma_{k,r}$ the K\"ahler potential defined in Proposition \ref{Prop:approx_solutions_aut}, such that the approximately extremal metric $\omega_{k,r}$ is given by $\omega_k+ i\del\delbar\gamma_{k,r}$. For each $k,r$ we define the map
\begin{equation*}
\begin{aligned}
\tau_{k,r} : \f{t} &\to C^\infty(X, \bb{R}) \\
\xi &\mapsto k\pi_Y^*h_B + h_q + \frac{1}{2}\langle \nabla\gamma_{k,r}, \nabla (k\pi_Y^*h_B+h_q) \rangle_{\omega_k},
\end{aligned}
\end{equation*}
where $h_B$ and $h_q$ are the holomorphy potentials defined in Remark \ref{Rmk:torus}. The map $\tau_{k,r}$ associates to a $T$-invariant holomorphic vector field the correspondent holomorphy potential with respect to $\omega_{k,r}$

We apply the theorem to the operators
\begin{equation*}
\begin{aligned}
&F_{k,r} : (L^2_{0,p+4})^T \times \overline{\f{t}}\to (L^2_{0,p})^T \\
&F_{k,r}(\varphi, h) = \mrm{Scal}(\omega_{k,r} + i \del\delbar\varphi)-\frac{1}{2}\langle \nabla\eta_{k,r}, \nabla \gamma_{k,r} \rangle -\eta_{k,r}- \frac{1}{2}\langle \nabla (\tau_{k,r}(h)), \nabla \varphi)\rangle - \tau_{k,r}(h),
\end{aligned}
\end{equation*}
where $\eta_{k,r}$ is the K\"ahler potential which makes $\omega_{k,r}$ approximately extremal.
The linearisation of $F_{k,r}$ is the operator
\begin{equation*}
\begin{aligned}
&G_{k,r} : (L^2_{0,p+4})^T\times \overline{\f{t}} \to (L^2_p)^T\\
&(\varphi, h) \mapsto -\m{D}^*_{k,r}\m{D}(\varphi) + \frac{1}{2}\langle \nabla(\mrm{Scal}(\omega_{k,r})-\tau_{k,r}(h)), \nabla \varphi\rangle-\tau_{k,r}(h).
\end{aligned}
\end{equation*}
The proof requires two steps: the first one is to ensure that the linearisation is an isomorphism with bounded inverse $P_{k,r}$. Theorem \ref{Thm:quantitative_implicit_function_theorem} then gives $\delta_k$ such that if $\norm{F_{k,r}(0)}<\delta_k$, a zero of $F_{k,r}$ exists.
Since we want to find a zero for all $k$, the second step is to find a value of $r$ for which the norm $\norm{F_{k,r}(0)}$ converges to zero quicker than $\delta_k$.
The first step is contained in the following lemma \cite[Lemma 6.6]{DervanSektnan_ExtremalFibrations}, based on \cite[Lemmas 6.5,6.6,6.7]{Fine_cscK_fibrations}.

\begin{lemma}
There exists a constant $C$ independent of $k$ such that $G_{k,r}$ has a right inverse $P_{k,r}$ such that
\begin{equation*}
\norm*{P_{k,r}} \le C k^{5/2}.
\end{equation*}
\end{lemma}

The second step relies on the following result \cite[Lemma 6.6]{DervanSektnan_ExtremalFibrations}, which is a consequence of the mean value theorem.
\begin{lemma}\label{Lemma:lipschitz}
Let $\m{N}_{k,r} = F_{k,r} - \dd_0F_{k,r}$ be the nonlinear part of the extremal operator. Then there are constant $c, C$ such that for all $r$ sufficiently large, if $f_i \in (L^2_{p+4})^T \times \overline{\f{t}}$ for $i=1,2$ satisfy $\norm{f_i}\le c$, then
\begin{equation*}
\norm*{\m{N}_{k,r}(f_1)-\m{N}_{k,r}(f_2)}_{L^2_{p}}\le C\left( \norm{f_1}_{L_{p+4}^2(\omega_{k,r})}+\norm{f_2}_{L^2_{p+4}(\omega_{k,r})}\right)\norm*{f_1-f_2}_{L^2_{p+4}(\omega_{k,r})}.
\end{equation*}
\end{lemma}

By applying the implicit function Theorem \ref{Thm:quantitative_implicit_function_theorem}, we can now complete the proof of Theorem \ref{Thm:existence_extremal_metrics} as follows.
Lemma \ref{Lemma:lipschitz} implies that $\m{N}_{k,r}$ is Lipschitz on any ball of radius $\rho$ sufficiently small, with Lipschitz constant $\rho C$.
Thus the radius $\delta'$ on which $\m{N}_{k,r}$ is Lipschitz with constant $(2\norm{P}_{k,r})^{-1}$ is bounded below by some multiple of $k^{-5/2}$.
Hence $\delta = \delta'(2\norm{P})^{-1}$ is bounded below by a multiple of $k^{-5}$.
In order to apply the implicit function theorem, it remains to bound $F_{k,r}(0,0)$. The point-wise bound $F_{k,r} = O(k^{(-r-1)/2})$ is provided by Proposition \ref{Prop:approx_solutions_aut}. Results of Fine \cite[Lemma 5.6, 5.7]{Fine_cscK_fibrations} can be applied directly to our situation in order to have a $L^2_{p}(\omega_{k,r})$-bound on $F_{k,r}(0)$ of order $k^{5-\frac{1}{2}}$, when $r>5$.
Thus the hypotheses of the implicit function theorem are satisfied and $\norm{F_{k,r}(0)}$ converges to zero quicker than $\delta_k$.

\bibliographystyle{plain} 
\bibliography{OSCbibliography}

\end{document}